\documentclass[reqno,12pt]{amsart}
\usepackage{latexsym}
\usepackage{amsmath,amssymb}
\usepackage{amsthm}
\usepackage{amsfonts}
\newtheorem{Lemma}{Lemma}[part]

\newtheorem{Remark}{Remark}[part]
\newtheorem{Theorem}{Theorem}[part]

\numberwithin{Assumption}{section} \numberwithin{Corollary}{section}
\numberwithin{Definition}{section} \numberwithin{equation}{section}
\numberwithin{Example}{section} \numberwithin{Lemma}{section}
\numberwithin{Proposition}{section} \numberwithin{Remark}{section}
\numberwithin{Theorem}{section}

\def\b{\beta}
\def\g{\gamma}

\def\i{\infty}

\newcommand{\R}{{\mathbb R}}
\newcommand{\T}{{\mathbb T}}

\date{}
\title{Global classical  solution of the Cauchy problem to 1D compressible Navier-Stokes equations with large initial data}
\author [Quansen Jiu, ~~ Mingjie Li~~and~~ Yulin Ye]{}

\date{}

\begin{document}
\maketitle

 \centerline{\scshape Quansen Jiu\footnote{The research is partially
supported by NSFC (No. 11171229, 11231006 and 11228102) and Project
of Beijing Chang Cheng Xue Zhe.}}

\medskip
{\footnotesize
  \centerline{School of Mathematical Sciences, Capital Normal University}
  \centerline{Beijing  100048, P. R. China}
   \centerline{  \it Email: jiuqs@mail.cnu.edu.cn}
   }
\vspace{3mm}
 \centerline{\scshape Mingjie Li\footnote{The research is supported by the NSFC
(No. 11201503). }}
\medskip
{\footnotesize
  \centerline{College of Science, Minzu University of China}
  \centerline{Beijing 100081, P. R. China}
   \centerline{  \it Email: lmjmath@gmail.com}
   }
\vspace{3mm}
  \centerline{\scshape Yulin Ye }

\medskip
{\footnotesize
  \centerline{School of Mathematical Sciences, Capital Normal University}
  \centerline{Beijing  100048, P. R. China}
   \centerline{  \it Email: nkyelin@163.com}
   }
\vspace{3mm}

\pagestyle{myheadings} \thispagestyle{plain}\markboth{\small{Q. S.
Jiu, M. J. Li, Y. L. Ye}} {\small{Cauchy problem of the 1-D
compressible N-S equations }}

\vspace{3mm}

\textbf{Abstract:}  In this paper, we prove that  the 1D Cauchy
problem of the compressible Navier-Stokes equations admits a unique
global
 classical solution $(\rho,\rm u)$ if the viscosity
$\mu(\rho)=1+\rho^{\beta}$ with $\beta\geq0$. The initial data can
be arbitrarily large and may contain vacuum.
 Some new weighted estimates  of the density and velocity are obtained  when deriving higher order estimates of the solution.


\section{Introduction and Main Results}
 \setcounter{equation}{0}
\setcounter{Assumption}{0} \setcounter{Theorem}{0}
\setcounter{Proposition}{0} \setcounter{Corollary}{0}
\setcounter{Lemma}{0}

In this paper, we consider the following compressible
Navier-Stokes equations with density-dependent viscosity
coefficients:

\begin{equation}\label{a1}
\left\{\begin{array}{lll}
\rho_{t}+(\rho u)_{x}=0,\\
(\rho u)_{t}+(\rho{u}^{2})_{x}+[p(\rho)]_{x}=[\mu(\rho)u_{x}]_{x},
\end{array}\right.
\end{equation}
where $t\geq0,x\in\R,\rho=\rho(x,t)$\ and\ $u=u(x,t)$ represent the
fluid density and velocity respectively and the pressure $P$ is
given by
\begin{equation}\label{a2}
P(\rho)=R\rho^{\g},\g>1.
\end{equation}
For simplicity, we assume that
\begin{equation}\label{a3}
\mu(\rho)=1+b\rho^{\beta},\ \ \b\geq0.
\end{equation}
In the sequel, we set $R=b=1$ without loss of  generality. We
consider the Cauchy problem for (\ref{a1}) with $(\rho,u)$ vanishing
at infinity. The initial data is imposed as
\begin{equation}\label{a4}
(\rho,u)|_{t=0}=(\rho_{0}(x),u_{0}(x)),\ \ \  x\in\R.
\end{equation}

It is known that in the presence of vacuum, the solution of the
compressible Navier-Stokes equations with constant viscosity will
behave singularly (see \cite{Xin1998},\cite{XY2013},  \cite{hs91})
in general.
 By some physical considerations, Liu, Xin and Yang in \cite{lxy98}  introduced the modified compressible
Navier-Stokes equations with density-dependent viscosity
coefficients for isentropic fluids. In fact, while deriving the
compressible Navier-Stokes equations from the Boltzmann equations by
the Chapman-Enskog expansions, the viscosity depends on the
temperature and correspondingly depends on the density for
isentropic cases. Moreover,  the viscous Saint-Venant system for the
shallow water equations, derived from the incompressible
Navier-Stokes equations with a moving free surface, corresponds a
kind of  compressible Navier-Stokes equations with density-dependent
viscosity (see \cite{GJX2008} and references therein).

 The one-dimensional compressible Navier-Stokes equations
with density-dependent viscosity have been widely studied (see
\cite{DWZ2011, Hoff1987, js98, JX2008, jwx1, jwx2, KS1977, MV2008,
SZ2003, YYZ2001, YZ2002} and references therein). However, the
global well-posedness of classical solutions with large initial data
in multi-dimensional case is completely open. Even the global
existence of weak solutions in multi-dimensional case remains open
except under spherically symmetric assumptions \cite{GJX2008}. In
\cite{VK1995}, Vaigant-Kazhikhov first proposed and studied the
following two-dimensional Navier-Stokes equations
\begin{equation}\label{kv} \left\{\begin{array}{lll}
\rho_t+\textrm{div}(\rho \mathbf{U})=0, \\
(\rho\mathbf{U})_t+\textrm{div}(\rho\mathbf{U}\otimes\mathbf{U})+\nabla
P(\rho)=\mu\triangle
\mathbf{U}+\nabla((\mu+\lambda(\rho))\textrm{div}\mathbf{U}).
\end{array}\right.
\end{equation}
 Here $\rho(x,t)$ and
$\mathbf{U}=(u_1(x,t), u_2(x,t))$ represent the density and velocity
of the flow respectively. It is assumed  in \cite{VK1995} that the
shear viscosity $\mu
> 0$ is a positive constant and the bulk viscosity satisfies $
\lambda(\rho)=\rho^\beta$ with $\beta> 0$ in general.  For the
periodic problem on the torus $\T^2$ and under assumptions that the
initial density is uniformly away from vacuum and $\beta> 3$ ,
Vaigant-Kazhikhov established the global well-posedness of the
classical  solution to (\ref{kv}) in \cite{VK1995}. Jiu-Wang-Xin
\cite{JWX2012} improved the result  and obtained the global
well-posedness of the classical solution with large initial data
permitting vacuum. Later on, Huang-Li relaxed the index $\beta$ to
be $\beta > 4/3$ and studied the large time behavior of the
solutions in \cite{hl121}. For the 2D Cauchy problems with vacuum
states at far fields, Jiu-Wang-Xin \cite{JWX2013} and Huang-Li
\cite{HL2012} independently considered the global well-posedness of
classical solution in different weighted spaces. Recently,
Jiu-Wang-Xin in \cite{jwx121} studied the global well-posedness to
the Cauchy problem with absence of vacuum   at far fields and proved
that if there is no vacuum initially then there will not appear
vacuum in any finite time.

In this paper, we will study the global existence and uniqueness of
classical solution to the one-dimensional Cauchy problem for  the
isentropic compressible Navier-Stokes equations
(\ref{a1})-(\ref{a4}).  The initial data is assumed to be large and
may contain vacuum. The index $\beta\ge 0$ is much more general in
comparison  with that in \cite{VK1995}, \cite{JWX2013},
\cite{HL2012} and \cite{jwx121} such that the constant viscosity is
permitted in our result. Note that for the initial data satisfying
(\ref{a5})-(\ref{a6}), the local existence and uniqueness of
classical solutions to (\ref{a1})-(\ref{a4}) have been established
in \cite{CK2006}, \cite{LZ2012}. Thus, to obtain the global
classical solution, one needs  to obtain a priori estimates of the
first and
 higher order derivatives of the solution. This
paper is motivated by \cite{DWZ2011} in which the one-dimensional
initial-boundary problem was studied and \cite{JWX2013} in which the
two-dimensional Cauchy problem of (\ref{kv}) was studied.  The upper
bound of the density will be proved in a new approach in this paper
and hence the a priori estimates of the first order derivative of
the solution can be obtained in a direct way.  However, when
deriving the estimates of the higher order derivatives of the
solution, new difficulty will be encountered in this paper since the
Poincare inequality can not be used and we have no any $L^p (1\le p
<\infty)$ estimates of the velocity.  To overcome this difficulty,
inspired by \cite{JWX2013}, we obtain some new weighted estimates on
the solution $(\rho, u)$ by using Cafferelli-Kohn-Nirenberg weighted
inequality and furthermore obtain some $L^p (1<p<\infty)$ and
$L^\infty$ estimates of the velocity (see \eqref{95-1}). Moreover,
by modifying the elegant estimates on the material derivatives of
the velocity developed by Hoff (\cite{hoff95}), the weighted spatial
estimates on both the gradient and the material derivatives of the
velocity are achieved. Based on these, one can obtain a priori
estimates of the higher order derivatives of the solution.

Denote the standard homogeneous and inhomogeneous Sobolev spaces as
follows:
$$D^k=\{u\in L^1_{loc}(\R)\big|\| \partial^k_x u\|_{L^2}<\i\} ,\ \ \ \|u\|_{D^k}=
\|\partial^k_x u\|_{L^2},\ \ \ H^k=L^2\cap D^k.$$

The main result of this paper can be stated as

\begin{Theorem}\label{theorem}
Suppose that the initial values $(\rho_{0},u_{0})(x)$ satisfy
\begin{equation}\label{a5}
\begin{aligned}
&0\leq(\rho_{0},\rho^{\beta}_{0},\rho^{\g}_{0})\in L^{1}(\R)\cap H^{2}(\R),\ \ u_{0}\in D^{1}(\R)\cap D^{2}(\R),\\
&\sqrt{\rho_{0}}u_{0}(1+|x|^{\frac{\alpha}{2}})\in L^{2}(\R),\ \
\partial_{x}{u_{0}}|x|^{\frac{\alpha}{2}}\in L^{2}(\R),\
(|x|^{\frac{\alpha}{2}}\rho_{0}^{\frac{\beta}{2}},|x|^{\frac{\alpha}{2}}\rho_{0}^{\frac{\gamma}{2}})\in
L^{2}(\R),
\end{aligned}
\end{equation}
for $\beta\geq0$ and
$2<\alpha<1+\frac{2}{\sqrt[3]{1+\sqrt[3]{4}}}$, and the
compatibility condition
\begin{equation}\label{a6}
[\mu(\rho_{0})u_{0x}]_{x}-[p(\rho_{0})]_{x}(x)=\sqrt{\rho_{0}}g(x),\
\ x\in \R,
\end{equation}
with some $g$ satisfying
$\sqrt{\rho_{0}}g(1+|x|^{\frac{\alpha}{2}})\in L^{2}(\R)$. Then for
any $T>0$, there exists a unique global classical solution
$(\rho,u)(t,x)$ to the Cauchy problem $(\ref{a1})-(\ref{a4})$,
satisfying
\begin{equation}\label{a7}
\begin{aligned}
&0\leq\rho\leq C,\ \ (\rho,\rho^{\gamma},\rho^{\beta})\in C([0,T];H^{2}(\R)),\ (\rho_{t},(\rho^{\gamma})_{t},(\rho^{\beta})_{t}\in C([0,T];H^{1}(\R)),\\
&\rho_{tt}\in C([0,T];L^{2}(\R)),((\rho^{\gamma})_{tt},(\rho^{\beta})_{tt})\in L^{\infty}([0,T];L^{2}(\R)),(\rho u)_{t}\in C([0,T];H^{1}(\R)),\\
&\sqrt{\rho}u(1+|x|^{\frac{\alpha}{2}}),\sqrt{\rho}{\dot{u}}(1+|x|^{\frac{\alpha}{2}}),u_{x}|x|^{\frac{\alpha}{2}} ,(|x|^{\frac{\alpha}{2}}\rho_{0}^{\frac{\beta}{2}},|x|^{\frac{\alpha}{2}}\rho_{0}^{\frac{\gamma}{2}})\in C([0,T];L^{2}(\R))\\
&u\in C([0,T];L^{\frac{2}{\alpha-1}}(\R)\cap D^{2}(\R))\cap L^{2}(0,T;L^{\frac{2}{\alpha-1}}\cap D^{3}(\R)),\sqrt{t}u\in L^{\infty}(0,T;D^{3}),\\
&u_{t}\in L^{\infty}(0,T;L^{\frac{2}{\alpha-1}}(\R)\cap D^{1}(\R)),\ \sqrt{t}u_{t}\in L^{2}(0,T;D^{2}(\R))\cap L^{\infty}(0,T;L^{\frac{4}{\alpha-2}}(\R)\cap D^{1}(\R)),\\
&tu_{t}\in L^{\infty}(0,T;D^{2}(\R)),\sqrt{t}\sqrt{\rho}u_{tt}|x|^{\frac{\alpha}{4}}\in L^{2}(0,T;L^{2}(\R)),t\sqrt{\rho}u_{tt}\in L^{\infty}(0,T;L^{2}(\R)),\\
&t\partial_{x}{u_{tt}}\in L^{2}(0,T;L^{2}(\R)),
\end{aligned}
\end{equation}
where $\dot{u}$ is the material derivative of $u$ defined as
$\dot{u}=(\partial_{t}+u\cdot\partial_{x})u$.
\end{Theorem}

\begin{Remark}
In our result, the index $\beta\ge 0$ is  more general  in
comparison with the results obtained in \cite{DWZ2011} for the
initial-boundary problem and in \cite{JWX2013} for the
two-dimensional Cauchy problem of (\ref{kv}). Much more general
viscosity $\mu(\rho)$ can be treated in a similar way.
\end{Remark}

The rest of the paper is organized as follows: In Section 2, we
collect some elementary  facts and  derive the  a priori estimates
of the solution which are needed to extend the local solution to  a
global one. In Section 3, we give the proof of  the main result.


\section{A priori estimates}
In this section,  we will establish various a priori estimates and
weighted estimates on classical solution $(\rho,u)$ on the interval
$[0,T]$ for any $T>0$. Before that, we give  the
Caffarelli-Kohn-Nirenberg weighted inequalities, which will be used
in the a priori estimates of the higher order derivatives of the
solution.

\begin{Lemma}[\bf Caffarelli-Kohn-Nirenberg weighted inequality \cite{CKN1984},
\cite{CW2001}]\label{lemmab2}

\ \\
(1) $\forall h\in C_{0}^{\infty}(\R)$, it holds that
\begin{equation}\label{b1}
\|{|x|^{\kappa}h}\|_{r}\leq C\|{|x|^{\alpha}|\partial_{x} h|}\|_{p}^{\theta}\|{|x|^{\beta}h}\|_{q}^{1-\theta}
\end{equation}
where $1\leq p,q<\infty,0<r<\infty,0\leq\theta\leq1,\frac{1}{p}+\alpha>0,\frac{1}{q}+\beta>0,\frac{1}{r}+\kappa>0$ and satisfying
\begin{equation}\label{b2}
\frac{1}{r}+\kappa=\theta(\frac{1}{p}+\alpha-1)+(1-\theta)(\frac{1}{q}+\beta)
\end{equation}
and\\
$$\kappa=\theta\sigma+(1-\theta)\beta$$\\
with $0\leq\alpha-\sigma \ if\ \theta>0\ \ and \ 0\leq\alpha-\sigma\leq1\ if\ \theta>0\ and\ \frac{1}{p}+\alpha-1=\frac{1}{r}+\kappa$.\\
(2)\ (Best constant for Caffarelli-Kohn-Nirenberg weighted
inequality) $\forall h\in C_{0}^{\infty}(\R)$, it holds that
\begin{equation}\label{b3}
\||x|^{b}h\|_{p}\leq C_{a,b}\||x|^{a}\partial_{x}{h}\|_{2}
\end{equation}
where $a>\frac{1}{2},a-1\leq b\leq a-\frac{1}{2}\ and\ p=\frac{2}{2(a-b)-1}.\ If\ b=a-1,then\ p=2$ and the best constant in the inequality $(\ref{b3})$ is
$$C_{a,b}=C_{a,a-1}=|\frac{2a-1}{2}|$$
The proof of (1) can be found in \cite{CKN1984} and the proof of (2)
can be found in \cite{CW2001}.
\end{Lemma}

\vspace{4mm}

\subsection{ {Uniform upper bound of the density}} \ \\

In this subsection, we will present a new approach to obtain the
upper bound of the density.  The following is the usual energy
estimate.

\begin{Lemma}\label{lemmac1}
Let $(\rho,u)$  be a smooth solution to $(\ref{a1})-(\ref{a4})$.
Then for any $T>0$, it holds
\begin{equation}\nonumber
\int_{\R}(\rho
u^{2}+\rho^{\gamma})(t)dx+\int_{0}^{T}\int_{\R}(1+\rho^\beta)(u_{x})^{2}dx\leq
C.
\end{equation}
\end{Lemma}

The upper bound of the density is stated as follows.
\begin{Lemma}\label{lemmac2}
Suppose that  $(\rho,u)$ is a smooth solution to
$(\ref{a1})-(\ref{a4})$. Then for any $T>0$, there exists an
absolute constant $C>0$ which depends on the initial data and
$\beta\ge 0$ such that
\begin{equation}\nonumber
\rho(x,t)\leq C, \ (x,t)\in \R\times(0,T].
\end{equation}

\begin{proof}
Let $$\xi=\int_{-\infty}^{x}\rho u(y)dy. $$\\
Using the momentum equation $(\ref{a1})_{2}$ , we  have\\
$$\xi_{tx}+(\rho u^{2})_{x}=(\mu(\rho)u_{x})_{x}-p_{x}.$$
Integrating with respect to $x$ over $(-\infty,x)$ yields
\begin{equation}\label{c3}
\xi_{t}+\rho u^{2}-\mu(\rho)u_{x}+p=0,
\end{equation}
Using the mass equation $(\ref{a1})_{1}$, we rewrite $(\ref{c3})$ as
\begin{equation}\label{c4}
\xi_{t}+\rho u^{2}+\mu(\rho)\frac{\rho_{t}+u\rho_{x}}{\rho}+p=0.
\end{equation}
Let $X(t,x)$ be the particle trajectory defined by
\begin{equation}\nonumber
\Biggl\{\begin{array}{c}
\frac{dX(t,x)}{dt}=u(X(t,x),t),\\[1mm]
X(0,x)=x.
\end{array}
\end{equation}
Then
\begin{equation}\label{e1}
\frac{d\xi}{dt}(X(t,x),t)=\xi_{t}+u\xi_{x}=\xi_{t}+\rho u^{2}.
\end{equation}
Denote  $$\eta(\rho)=\int_{1}^{\rho}\frac{\mu(s)}{s}ds =\Biggl\{\begin{array}{c}\ln\rho+\frac{1}{\beta}(\rho^{\beta}-1),if \ \beta>0,\\
2\ln\rho ,\ \ \ \ \ \ \ \ \ \ \ \ \ \ if \ \beta=0.\end{array}$$ It
follows from $(\ref{c4})$ and $(\ref{e1})$ that
\begin{equation}\label{c6}
\frac{d}{dt}(\xi+\eta)(X(t,x),t)\leq\frac{d}{dt}(\xi+\eta)(X(t,x),t)+p(X(t,x),t)=0.
\end{equation}
Integrating $(\ref{c6})$ over $(0,t)$, we have
\begin{equation}\nonumber
(\xi+\eta)(X(t,x),t)\leq\xi(X(0,x),0)+\eta(X(0,x),0).
\end{equation}
Since $$\xi(X(0,x),0)=\int_{-\infty}^{X((0,x),0)}\rho_{0}u_{0}(y)dy\leq |\int_{\R}\rho_{0}u_{0}dy| \leq\|\sqrt{\rho_{0}}u_{0}\|_{L^{2}}\|\rho_{0}\|_{L^{1}}^{1/2}\leq C,$$\\
\begin{equation}\nonumber
\begin{aligned}
\eta(X(0,x),0)=\int_{1}^{\rho_{0}}\frac{\mu(s)}{s}ds
&=\Biggl\{\begin{array}{c}\ln\rho_{0}+\frac{1}{\beta}(\rho_{0}^{\beta}-1),\ if \ \beta>0\\
2\ln\rho_{0} ,\ \ \ \ \ \ \ \ \ \ \ \ \ \ \ if \ \beta=0\end{array}\\
&\leq\Biggl\{\begin{array}{c}\rho_{0}+\frac{1}{\beta}(\rho_{0}^{\beta}-1),\
\ \ \ if \ \beta>0\\2\rho_{0},\ \ \ \ \ \ \ \ \ \ \ \ \ \ \ \ \ \ \
if \ \beta=0\end{array},
\end{aligned}
\end{equation}
we have$$\xi(x,t)+\eta(x,t)\leq C,$$ which implies
\begin{equation}\nonumber
\ln \rho+\frac{1}{\beta}(\rho^{\beta}-1)\leq C-\int_{-\infty}^{x}\rho udx\leq C+\int_{\R}|\rho u|dx\leq C +\|\sqrt{\rho}u\|_{L^{2}}\|\rho\|_{L^{1}}\leq C,\  \beta>0,
\end{equation}
or
$$2\ln\rho\leq C-\int_{-\infty}^{x}\rho udx\leq C+\int_{\R}|\rho u|dx\leq C +\|\sqrt{\rho}u\|_{L^{2}}\|\rho\|_{L^{1}}\leq C,\  \beta=0.$$
Then we  have$$\ln \rho \leq \Biggl\{\begin{array}{c}C+\frac{1}{\beta},if\ \beta>0,\\
C,\ \ \ \ \ \ if\ \beta=0.\end{array}$$
Consequently
$$\rho (x,t)\leq C,\ \beta\geq0.$$\\
The proof of the lemma  is completed.
\end{proof}
\end{Lemma}

\vspace{3mm}

\subsection{{The  estimates of the first derivatives}}\ \\

The first derivative estimates of the velocity is as follows.
\begin{Lemma}\label{lemmac3}
Suppose that $(\rho,u)$  is a smooth solution to
$(\ref{a1})-(\ref{a4})$. Then for any $T>0$, it holds
\begin{equation}\nonumber
\int_{\R}(u_{x})^{2}dx+\int_{0}^{T}\int_{\R}\rho u_{t}^{2}dxdt+\int_{0}^{T}\|u_{x}\|_{L^{\infty}}^{2}dt\leq C(T).
\end{equation}
\begin{proof}
Using $(\ref{a1})_{1}$,we rewrite $(\ref{a1})_{2}$ as
\begin{equation}\label{c7}
\rho u_{t}+\rho uu_{x}+(\rho^{\gamma})_{x}=(\mu(\rho)u_{x})_{x}.
\end{equation}
Multiplying on both sides of $(\ref{c7})$ by $u_{t}$, integrating
over $\R$, we have
\begin{equation}\nonumber
\begin{aligned}
\int_{\R}\rho u_{t}^{2}dx&+\frac{1}{2}\frac{d}{dt}\int_{\R}\mu(\rho)u_{x}^{2}dx=\frac{1}{2}\int_{\R}(\rho^{\beta})_{t}u_{x}^{2}dx-\int_{\R}\rho u_{t}uu_{x}dx-\int_{\R}(\rho^{\gamma})_{x}u_{t}dx\\
&\leq 
\frac{d}{dt}\int_{\R}\rho^{\gamma}u_{x}dx+\gamma\int_{\R}\rho^{\gamma-1}(\rho u_{x}+u\rho_{x})u_{x}dx-\frac{1}{2}\int_{\R}(u(\rho^{\beta})_{x}+\beta\rho^{\beta}u_{x})u_{x}^{2}dx\\
&+\frac{1}{2}\int_{\R}\rho u_{t}^{2}dx+C\int_{\R}\rho
u^{2}u_{x}^{2}dx.
\end{aligned}
\end{equation}
It yields
\begin{equation}\label{c8}
\begin{aligned}
\frac{1}{2}\int_{\R}\rho u_{t}^{2}dx+\frac{1}{2}\frac{d}{dt}\int_{\R}\mu(\rho)u_{x}^{2}dx&\leq\frac{d}{dt}\int_{\R}\rho^{\gamma}u_{x}dx+\gamma\int_{\R}\rho^{\gamma-1}(\rho u_{x}+u\rho_{x})u_{x}dx\\
&-\frac{1}{2}\int_{\R}(u(\rho^{\beta})_{x}+\beta\rho^{\beta}u_{x})u_{x}^{2}dx+C\int_{\R}\rho u^{2}u_{x}^{2}dx\\
&=\frac{d}{dt}\int_{\R}\rho^{\gamma}u_{x}dx+\Sigma_{i=1}^{2}I_{i}+C\int_{\R}\rho
u^{2}u_{x}^{2}dx,
\end{aligned}
\end{equation}
where we used the following equation
$$(\rho^{\beta})_{t}+u(\rho^{\beta})_{x}+\beta\rho^{\beta}u_{x}=0,\ \beta\geq0 .$$
Then we estimate the  terms $I_{1}-I_{2}$ as follows.
\begin{equation}\label{c9}
\begin{aligned}
&I_{1}=\gamma\int_{\R}\rho^{\gamma}u_{x}^{2}dx+\gamma\int_{\R}\rho^{\gamma-1}\rho_{x}uu_{x}dx\\
&\leq 
C\|\rho\|_{L^{\infty}}^{\gamma}\int_{\R}u_{x}^{2}dx+\gamma\int_{\R}(\int_{0}^{\rho}\frac{s^{\gamma-1}}{\mu(s)}ds)_{x}u(\mu(\rho)u_{x}-\rho^{\gamma})dx+
\gamma\int_{\R}(\int_{0}^{\rho}\frac{s^{2\gamma-1}}{\mu(s)}ds)_{x}udx\\
&\leq
C\int_{\R}\mu(\rho)u_{x}^{2}dx+C(T)-\gamma\int_{\R}(\int_{0}^{\rho}\frac{s^{\gamma-1}}{\mu(s)}ds)
u(\rho u_{t}+\rho uu_{x})dx\\
&\ \ \ \ \ +C\int_{\R}(\int_{0}^{\rho}s^{2\gamma-1}ds)|u_{x}|dx\\
&\leq C\int_{\R}\mu(\rho)u_{x}^{2}dx+\varepsilon\int_{\R}\rho
u_{t}^{2}dx+C\int_{\R}\rho u^{2}u_{x}dx+C(T).
\end{aligned}
\end{equation}

\begin{equation}\label{c13}
\begin{aligned}
&I_{2}=-\frac{1}{2}\int_{\R}(u(\rho^{\beta})_{x}+\beta\rho^{\beta}u_{x})u_{x}^{2}dx=-\frac{1}{2}\int_{\R}\beta\rho^{\beta}u_{x}^{3}dx-\frac{1}{2}\int_{\R}\beta\rho^{\beta-1}\rho_{x}uu_{x}^{2}dx\\
&=I_{21}+I_{22}.
\end{aligned}
\end{equation}
Direct estimates give
\begin{equation}\label{c13-1}
\begin{aligned}
I_{21}=-\frac{1}{2}\int_{\R}\beta\rho^{\beta}u_{x}^{3}dx\leq
C\|u_{x}\|_{L^{\infty}}\int_{\R}\mu(\rho)u_{x}^{2}dx
\end{aligned}
\end{equation}
\begin{equation}\label{c13-2}
\begin{aligned}
&I_{22}=-\frac{1}{2}\int_{\R}\beta\rho^{\beta-1}\rho_{x}uu_{x}^{2}dx\\
&=-\frac{1}{2}\int_{\R}\frac{\beta\rho^{\beta-1}\rho_{x}}{(\mu(\rho))^{2}}u(\mu(\rho)u_{x}-\rho^{\gamma})^{2}dx -\int_{\R}\frac{\beta\rho^{\beta-1}\rho_{x}}{(\mu(\rho))^{2}}u\mu(\rho)u_{x}\rho^{\gamma}dx\\
&\ \ \ \ \ \ \ \ \ +\frac{1}{2}\int_{\R}\frac{\beta\rho^{2\gamma+\beta-1}}{(\mu(\rho))^{2}}\rho_{x}udx\\
&\leq -\frac{1}{2}\int_{\R}(\int_{0}^{\rho}\frac{\beta s^{\beta-1}}{(\mu(s))^{2}}ds)_{x}u(\mu(\rho)u_{x}-\rho^{\gamma})^{2}dx-\int_{\R}(\int_{0}^{\rho}\frac{\beta s^{\gamma+\beta-1}}{(\mu(s))^{2}}ds)_{x}u[\mu(\rho)u_{x}-\rho^{\gamma}]dx\\
&\ \ \ \ \ \ \ \ \ -\frac{1}{2}\int_{\R}(\int_{0}^{\rho}\frac{\beta s^{2\gamma+\beta-1}}{(\mu(s))^{2}}ds)_{x}udx\\
&\leq \frac{1}{2}\int_{\R}(\int_{0}^{\rho}\frac{\beta s^{\beta-1}}{(\mu(s))^{2}}ds)[u_{x}(\mu(\rho)u_{x}-\rho^{\gamma})^{2}+2u(\mu(\rho)u_{x}-\rho^{\gamma})(\mu(\rho)u_{x}-\rho^{\gamma})_{x}]dx\\
&\ \ \ \ \ \ \ \ \ +\int_{\R}(\int_{0}^{\rho}\frac{\beta s^{\gamma+\beta-1}}{(\mu(s))^{2}}ds)[u_{x}(\mu(\rho)u_{x}-\rho^{\gamma})+u(\mu(\rho)u_{x}-\rho^{\gamma})_{x}]dx\\
&\ \ \ \ \ \ \ \ \ +\frac{1}{2}\int_{\R}(\int_{0}^{\rho}\frac{\beta s^{2\gamma+\beta-1}}{(\mu(s))^{2}}ds)u_{x}dx\\
&\leq C(T)+C\int_{\R}\mu(\rho)u_{x}^{2}dx-\int_{\R}(\int_{0}^{\rho}\frac{\beta s^{\beta-1}}{(\mu(s))^{2}}ds)u(\mu(\rho)u_{x}-\rho^{\gamma})(\rho u_{t}+\rho uu_{x})dx\\
&\ \ \ \ \ \ \ \ \ -\int_{\R}(\int_{0}^{\rho}\frac{\beta s^{\gamma+\beta-1}}{(\mu(s))^{2}}ds)u(\rho u_{t}+\rho uu_{x})dx+\frac{1}{2}\int_{\R}\beta\frac{1}{2\gamma+\beta}\rho^{2\gamma+\beta}u_{x}dx\\
&\leq C(T)+C\int_{\R}\mu(\rho)u_{x}^{2}dx+\varepsilon\int_{\R}\rho
u_{t}^{2}dx+C\int_{\R}\rho u^{2}u_{x}^{2}dx+C\int_{\R}\rho
u^{2}u_{x}dx.
\end{aligned}
\end{equation}

Using the Gagliardo-Nirenberg inequality, we get
\begin{equation}\nonumber
\begin{aligned}
\|u_{x}\|_{L^{\infty}}
&\leq\|\mu(\rho)u_{x}-\rho^{\gamma}\|_{L^{\infty}}+\|\rho^{\gamma}\|_{L^{\infty}}\leq C\|\mu(\rho)u_{x}-\rho^{\gamma}\|_{L^{\infty}}+C\\
&\leq \|\mu(\rho)u_{x}-\rho^{\gamma}\|_{L^{2}}^{\frac{1}{2}}\|(\mu(\rho)u_{x}-\rho^{\gamma})_{x}\|_{L^{2}}^{\frac{1}{2}}+C\\
&\leq C(\|\sqrt{\mu(\rho)}u_{x}\|_{L^{2}}+1)^{\frac{1}{2}}(\|\sqrt{\rho}u_{t}\|_{L^{2}}+\|u_{x}\|_{L^{\infty}}\|\sqrt{\rho}u\|_{L^{2}})^{\frac{1}{2}}+C\\
&\leq C(\|\sqrt{\mu(\rho)}u_{x}\|_{L^{2}}+1)^{\frac{1}{2}}\|\sqrt{\rho}u_{t}\|_{L^{2}}^{\frac{1}{2}}+C(\|\sqrt{\mu(\rho)}u_{x}\|_{L^{2}}+1)^{\frac{1}{2}}\|u_{x}\|_{L^{\infty}}^{\frac{1}{2}}+C\\
&\leq \frac{1}{2}\|u_{x}\|_{L^{\infty}}+C(\|\sqrt{\mu(\rho)}u_{x}\|_{L^{2}}+1)+\varepsilon\|\sqrt{\rho}u_{t}\|_{L^{2}}.
\end{aligned}
\end{equation}
It concludes that
\begin{equation}\label{c18}
\|u_{x}\|_{L^{\infty}}\leq C\|\sqrt{\mu(\rho)}u_{x}\|_{L^{2}}+\varepsilon\|\sqrt{\rho}u_{t}\|_{L^{2}}+C.
\end{equation}
It follows that
\begin{equation}\label{c19}
\int_{\R}\rho u^{2}u_{x}dx\leq C\|u_{x}\|_{L^{\infty}}\int_{\R}\rho
u^{2}dx\leq C\|u_{x}\|_{L^{\infty}}\leq
C\|\sqrt{\mu(\rho)}u_{x}\|_{L^{2}}+\varepsilon\|\sqrt{\rho}u_{t}\|_{L^{2}}+C(T),
\end{equation}
and
\begin{equation}\label{c20}
\int_{\R}\rho u^{2}u_{x}^{2}dx\leq
\|u_{x}\|_{L^{\infty}}^{2}\int_{\R}\rho u^{2}dx\leq
C\|\sqrt{\mu(\rho)}u_{x}\|_{L^{2}}^{2}+\varepsilon\|\sqrt{\rho}u_{t}\|_{L^{2}}^{2}+C(T).
\end{equation}
Substituting $(\ref{c9})-(\ref{c20})$ into $(\ref{c8})$, we get
\begin{equation}\label{c21}
\frac{d}{dt}\int_{\R}\mu(\rho)u_{x}^{2}dx+\int_{\R}\rho
u_{t}^{2}dx\leq
\frac{d}{dt}\int_{\R}\rho^{\gamma}u_{x}dx+C(\int_{\R}\mu(\rho)u_{x}^{2}dx)^{2}+C(T).
\end{equation}
Integrating $(\ref{c21})$ over $[0,T]$ and using the Cauchy
inequality and Gronwall inequality, we have
\begin{equation}\nonumber
\int_{\R}\mu(\rho)u_{x}^{2}dx+\int_{0}^{T}\int_{\R}\rho
u_{t}^{2}dxdt\leq C(T).
\end{equation}
Using $(\ref{c18})$ again yields
$$\int_{0}^{T}\|u_{x}\|_{L^{\infty}}^{2}dt\leq C(T).$$
The proof of the lemma  is completed.
\end{proof}
\end{Lemma}

Next we show the  first derivative estimates of the density.
\begin{Lemma}\label{lemmac4}
Let $(\rho,u)$  be the smooth solution to $(\ref{a1})-(\ref{a4})$.
Then for any $T>0$, it holds
\begin{equation}\nonumber
\int_{\R}\bigg(\rho_{x}^{2}+(\rho^{\gamma})_{x}^{2}+(\rho^{\beta})_{x}^{2}\bigg)dx\leq C(T).
\end{equation}
\begin{proof}
Differentiating $(\ref{a1})_{1}$ with respect to $x$, multiplying
the resulting equation by $\rho_{x}$ and integrating over $\R$, we
have
\begin{equation}\label{c22}
\begin{aligned}
\frac{1}{2}\frac{d}{dt}\int_{\R}\rho_{x}^{2}dx&=-\frac{3}{2}\int_{\R}\rho_{x}^{2}u_{x}dx-\int_{\R}\rho u_{xx}\rho_{x}dx\\
&\leq C\|u_{x}\|_{L^{\infty}}\int_{\R}\rho_{x}^{2}+C\|\rho\|_{L^{\infty}}\|u_{xx}\|_{L^{2}}\|\rho_{x}\|_{L^{2}}.
\end{aligned}
\end{equation}
Since $$\rho u_{t}+\rho
uu_{x}+(\rho^{\gamma})_{x}=(\mu(\rho)u_{x})_{x},$$ combining
\eqref{c18} and Lemma \ref{lemmac3}, we have
\begin{equation}\label{c23}
\begin{aligned}
\|u_{xx}\|_{L^{2}}&\leq C(\|\rho u_{t}\|_{L^{2}}+\|\rho uu_{x}\|_{L^{2}}+\|(\rho^{\gamma})_{x}\|_{L^{2}}+\|(\rho^{\beta})_{x}u_{x}\|_{L^{2}})\\
&\leq C(\|\sqrt{\rho}u_{t}\|_{L^{2}}+\|u_{x}\|_{L^{\infty}}+\|(\rho^{\gamma})_{x}\|_{L^{2}}
+\|u_{x}\|_{L^{\infty}}\|(\rho^{\beta})_{x}\|_{L^{2}})\\
&\leq
C\bigg[\|\sqrt{\rho}u_{t}\|_{L^{2}}+1+\|(\rho^{\gamma})_{x}\|_{L^{2}}+(\|\sqrt{\rho}u_{t}\|_{L^{2}}+1)\|(\rho^{\beta})_{x}\|_{L^{2}}\bigg].
\end{aligned}
\end{equation}
Putting \eqref{c18}, \eqref{c23} into \eqref{c22}, we get
\begin{equation}\label{c24}
\begin{aligned}
&\frac{d}{dt}\int_{\R}\rho_{x}^{2}dx\leq
C(\|\sqrt{\rho}u_{t}\|_{L^{2}}+1)\int_{\R}\rho_{x}^{2}dx
+C\bigg[\|\sqrt{\rho}u_{t}\|_{L^{2}}+1+\|(\rho^{\gamma})_{x}\|_{L^{2}}\\
&\ \ \ \ +(\|\sqrt{\rho}u_{t}\|_{L^{2}}+1)\|(\rho^{\beta})_{x}\|_{L^{2}}\bigg]\|\rho_{x}\|_{L^{2}}\\
&\leq
C(\|\sqrt{\rho}u_{t}\|_{L^{2}}^{2}+1)\int_{\R}\rho_{x}^{2}dx+C(\|(\rho^{\gamma})_{x}\|_{L^{2}}^{2}
+\|(\rho^{\beta})_{x}\|_{L^{2}}^{2})+C(\|\sqrt{\rho}u_{t}\|_{L^{2}}^{2}+1)\\
&\leq
C(\|\sqrt{\rho}u_{t}\|_{L^{2}}^{2}+1)\int_{\R}\rho_{x}^{2}+(\rho^{\gamma})_{x}^{2}+(\rho^{\beta})_{x}^{2}dx
+C(\|\sqrt{\rho}u_{t}\|_{L^{2}}^{2}+1).
\end{aligned}
\end{equation}
Note that
$$(\rho^{\gamma})_{t}+u(\rho^{\gamma})_{x}+\gamma(\rho^{\gamma})u_{x}=0$$
and$$(\rho^{\beta})_{t}+u(\rho^{\beta})_{x}+\beta(\rho^{\beta})u_{x}=0,$$
for any $\gamma>1,\ \beta\geq 0$. We can obtain in a similar way
that
\begin{equation}\label{c24-1}
\begin{aligned}
&\frac{d}{dt}\int_{\R}(\rho^{\gamma})_{x}^{2}+(\rho^{\beta})_{x}^{2}dx\\
&\leq
C(\|\sqrt{\rho}u_{t}\|_{L^{2}}^{2}+1)\int_{\R}\rho_{x}^{2}+(\rho^{\gamma})_{x}^{2}+(\rho^{\beta})_{x}^{2}dx
+C(\|\sqrt{\rho}u_{t}\|_{L^{2}}^{2}+1).
\end{aligned}
\end{equation}
Substituting \eqref{c24-1} into \eqref{c24}, using Lemma
\ref{lemmac3} and the Gronwall inequality , we have
$$\int_{\R}\bigg(\rho_{x}^{2}+(\rho^{\gamma})_{x}^{2}+(\rho^{\beta})_{x}^{2}\bigg)dx\leq
C(T).$$ The proof of the lemma is completed.
\end{proof}
\end{Lemma}

\subsection{{Weighted energy estimates}}\ \\
In this subsection, we will establish the weighted energy estimates.
These will be used in estimates of the higher derivatives of the
velocity.
\begin{Lemma}\label{lemmac5}
Let $(\rho,u)$  be the smooth solution to $(\ref{a1})-(\ref{a4})$.
Then for any $T>0$ and $\alpha>0$, it holds
$$\int_{\R}\rho |u|^{\alpha+2}dx+\int_{0}^{T}\int_{\R}\mu(\rho)u_{x}^{2}|u|^{\alpha}dxdt\leq C(T).$$
\begin{proof}
For any $\alpha>0$, multiplying  $(\ref{a1})_{2}$ by
$(\alpha+2)|u|^{\alpha}u$ and integrating with respect to $x$ over
$\R$ yields that
\begin{equation}\nonumber
\begin{aligned}
\frac{d}{dt}&\int_{\R}\rho |u|^{\alpha+2}dx+(\alpha+2)(\alpha+1)\int_{\R}\mu(\rho)|u|^{\alpha}u_{x}^{2}dx\\
&=\int_{\R}\rho_{t}|u|^{\alpha+2}dx-(\alpha+2)\int_{\R}\rho |u|^{\alpha+2}u_{x}dx-\int_{\R}(\rho^{\gamma})_{x}(\alpha+2)|u|^{\alpha}udx\\
&=(\alpha+2)(\alpha+1)\int_{\R}\rho^{\gamma}|u|^{\alpha}u_{x}dx\\
&=(\alpha+2)(\alpha+1)\int_{\R}(|u|^{\frac{\alpha}{2}}u_{x})(\rho |u|^{\alpha+2})^{\frac{\alpha}{2(\alpha+2)}}\rho^{\gamma-\frac{\alpha}{2(\alpha+2)}}dx\\
&\leq \varepsilon\int_{\R}|u|^{\alpha}u_{x}^{2}dx+C\int_{\R}\rho
|u|^{\alpha+2}dx+C.
\end{aligned}
\end{equation}
Use the Gronwall inequality to get
$$\int_{\R}\rho |u|^{\alpha+2}dx+\int_{0}^{T}\int_{\R}\mu(\rho)u_{x}^{2}|u|^{\alpha}dxdt\leq C(T).$$
The proof of the lemma is completed.
\end{proof}
\end{Lemma}

\begin{Lemma}\label{lemmac6}
Let $(\rho,u)$  be the smooth solution to $(\ref{a1})-(\ref{a4})$.
Then for any $T>0$ and
$2<\alpha<1+\frac{2}{\sqrt[3]{1+\sqrt[3]{4}}}$, it holds
$$\int_{\R}|x|^{\alpha}(\rho u^{2}+\rho^{\gamma}+\rho^{\beta})dx+\int_{0}^{T}\int_{\R}|x|^{\alpha}\mu(\rho)u_{x}^{2}dxdt\leq C(T).$$
\begin{proof}
Multiplying  $(\ref{a1})_{2}$ by $|x|^{\alpha}u$ and integrating
with respect to $x$ over $\R$ yields that
\begin{equation}\label{c26}
\begin{aligned}
\frac{d}{dt}&\int_{\R}|x|^{\alpha}(\frac{1}{2}\rho u^{2}+\frac{1}{\gamma-1}\rho^{\gamma}+\rho^{\beta})dx+ \int_{\R}\mu(\rho)|x|^{\alpha}u_{x}^{2}dx\\
&= \frac{1}{2}\int_{\R}\alpha \rho u^{3}|x|^{\alpha-2}xdx+\frac{\gamma\alpha}{\gamma-1}\int_{\R}\rho^{\gamma}|x|^{\alpha-2}xudx-\alpha\int_{\R}\mu(\rho)|x|^{\alpha-2}xuu_{x}dx\\
&-(\beta-1)\int_{\R}|x|^{\alpha}\rho^{\beta}
u_{x}dx+\alpha\int_{\R}\rho^{\beta}|x|^{\alpha-2}xudx\\
&\equiv\Sigma_{i=1}^{5}J_{i}.
\end{aligned}
\end{equation}
The terms $J_{i}\ (i=1,2\cdot\cdot\cdot5)$ on the right side of
$(\ref{c26})$ are estimated as follows. By the H\"{o}lder
inequality, Caffarelli-Kohn-Nirenberg weighted inequality and Young
inequality, it holds that
\begin{equation}\label{c27}
\begin{aligned}
J_{1}&\leq C\int_{\R}\rho u^{3}|x|^{\alpha-1}dx=C\int_{\R}(\rho u^{2}|x|^{\alpha})^{\frac{\alpha-1}{\alpha}}\rho^{\frac{1}{\alpha}}u^{3-\frac{2(\alpha-1)}{\alpha}}dx\\
&=C\int_{\R}(\rho u^{2}|x|^{\alpha})^{\frac{\alpha-1}{\alpha}}(\rho u^{\alpha+2})^{\frac{1}{\alpha}}dx\\
&\leq \||x|^{\frac{\alpha}{2}}\sqrt{\rho}u\|_{L^{2}}^{\frac{2(\alpha-1)}{\alpha}}\|\sqrt{\rho}u^{\frac{\alpha}{2}+1}\|_{L^{2}}^{\frac{2}{\alpha}}\\
&\leq C\int_{\R}|x|^{\alpha}\rho u^{2}dx+C(T),
\end{aligned}
\end{equation}
$J_{2}$ can be estimated as
\begin{equation}\label{c28}
\begin{aligned}
J_{2}&\leq C\int_{\R}\rho^{\gamma}|u||x|^{\alpha-1}dx=C\int_{\R}(\sqrt{\rho^{\gamma}}|x|^{\frac{\alpha}{2}})(|u||x|^{\frac{\alpha}{2}-1})\sqrt{\rho^{\gamma}}dx\\
&\leq C\biggl(\int_{\R}\rho^{\gamma}|x|^{\alpha}dx\biggl)^{\frac{1}{2}}\parallel\rho\parallel_{L^{\infty}}^{\frac{\gamma}{2}}\parallel|u||x|^{\frac{\alpha}{2}-1}\parallel_{L^{2}}\\
&\leq C\biggl(\int_{\R}\rho^{\gamma}|x|^{\alpha}dx\biggl)^{\frac{1}{2}}\parallel|u_{x}||x|^{\frac{\alpha}{2}}\|_{L^{2}}\\
&\leq
\varepsilon\int_{\R}\mu(\rho)|x|^{\alpha}u_{x}^{2}dx+C\int_{\R}\rho^{\gamma}|x|^{\alpha}dx,
\end{aligned}
\end{equation}
where the index $\alpha$  satisfies
\begin{equation}\nonumber
\frac{1}{2}+\frac{\frac{\alpha}{2}-1}{1}>0\Longrightarrow\alpha>1.
\end{equation}
$J_3$ can be rewritten as
\begin{equation}\label{c29}
\begin{aligned}
J_{3}&=-\alpha\int_{\R}\mu(\rho)|x|^{\alpha-2}xuu_{x}dx\\
&=-\alpha\int_{\R}|x|^{\alpha-2}xuu_{x}dx-\alpha\int_{\R}\rho^{\beta}|x|^{\alpha-2}xuu_{x}dx\\
&\equiv J_{31}+J_{32}.
\end{aligned}
\end{equation}
Direct estimates give
\begin{equation}\label{c30}
\begin{aligned}
J_{31}&=-\alpha\int_{\R}|x|^{\alpha-2}xuu_{x}dx=-\alpha\int_{\R}\frac{1}{2}|x|^{\alpha-2}x(u^{2})_{x}dx\\
&=\frac{\alpha}{2}\int_{\R}u^{2}\biggl((\alpha-2)|x|^{\alpha-3}\frac{x}{|x|}x+|x|^{\alpha-2}\biggl)dx\\
&=\frac{\alpha}{2}\int_{\R}u^{2}\biggl((\alpha-1)|x|^{\alpha-2}\biggl)dx\leq\frac{\alpha(\alpha-1)}{2}\parallel|x|^{\frac{\alpha}{2}-1}u\parallel_{L^{2}}^{2}\\
&\leq \frac{\alpha(\alpha-1)^{3}}{8}\parallel|x|^{\frac{\alpha}{2}}u_{x}\parallel_{L^{2}}^{2}.
\end{aligned}
\end{equation}
The weight index $\alpha>0$ will be chosen to satisfy
\begin{equation}\label{a1+}
\frac{\alpha(\alpha-1)^{3}}{8}<1.
\end{equation}
$J_{32}$ is estimated as
\begin{equation}\label{c31}
\begin{aligned}
J_{32}&\leq
C\int_{\R}\rho^{\beta}|x|^{\alpha-1}|u||u_{x}|dx=C\int_{\R}(|x|^{\frac{\alpha}{3}-1}|u|)(|x|^{\frac{\alpha}{2}}|u_{x}|)(\rho^{\beta}|x|^{\alpha})
^{\frac{1}{6}}\rho^{\frac{5\beta}{6}}dx\\
&\leq C\parallel|x|^{\frac{\alpha}{3}-1}u\parallel_{L^{3}}\parallel|x|^{\frac{\alpha}{2}}u_{x}\parallel_{L^{2}}\parallel\rho^{\beta}|x|^{\alpha}\parallel_{L^{1}}
^{\frac{1}{6}}\\
&\leq C\parallel u_{x}\parallel_{L^{2}}^{1-\theta}\parallel|x|^{\frac{\alpha}{2}}u_{x}\parallel_{L^{2}}^{\theta}\parallel|x|^{\frac{\alpha}{2}}u_{x}\parallel_{L^{2}}\parallel\rho^{\beta}|x|^{\alpha}\parallel_{L^{1}}^{\frac{1}{6}}\\
&\leq C\parallel|x|^{\frac{\alpha}{2}}u_{x}\parallel_{L^{2}}^{1+\theta}\parallel\rho^{\beta}|x|^{\alpha}\parallel_{L^{1}}^{\frac{1}{6}}\\
&\leq
\varepsilon\parallel|x|^{\frac{\alpha}{2}}u_{x}\parallel_{L^{2}}^{2}+C\parallel\rho^{\beta}|x|^{\alpha}\parallel_{L^{1}}+C(T),
\end{aligned}
\end{equation}
where  $\theta\in(0,1)$ and $\alpha>1$ are chosen to satisfy
$$\frac{1}{3}+\frac{\alpha}{3}-1=(\frac{1}{2}-1)(1-\theta)+(\frac{1}{2}+\frac{\alpha}{2}-1)\theta,\  \ \frac{1}{3}+\frac{\alpha}{3}-1>0,$$
which implies
\begin{equation}\label{a3+}
\theta=\frac{2\alpha-1}{3\alpha}<\frac{2}{3},\ \alpha>2.
\end{equation}
By \eqref{a1+} and \eqref{a3+}, we first choose $\alpha$ as
\begin{equation}\nonumber
(\alpha-1)^{3}<\frac{8}{\alpha}<4\Longrightarrow\alpha<1+\sqrt[3]{4}.
\end{equation}
Then, to guarantee \eqref{a1+}, we impose
$$\frac{\alpha(\alpha-1)^{3}}{8}<\frac{(1+\sqrt[3]{4})(\alpha-1)^{3}}{8}<1,$$
which implies that
\begin{equation}\label{a5+}
(\alpha-1)^{3}<\frac{8}{1+\sqrt[3]{4}}\Longrightarrow\alpha<1+\frac{2}{\sqrt[3]{1+\sqrt[3]{4}}}.
\end{equation}
 Combining \eqref{a3+} and \eqref{a5+}, the index $\alpha$ is
 chosen to satisfy
\begin{equation}\label{a6+}
2<\alpha<1+\frac{2}{\sqrt[3]{1+\sqrt[3]{4}}}. \end{equation}
Concerning $J_4$ and $J_5$, we have
\begin{equation}\label{c32}
\begin{aligned}
J_{4}+J_{5}&\leq C\int_{\R}|x|^{\alpha}\rho^{\beta}|u_{x}|dx+C\int_{\R}\rho^{\beta}|x|^{\alpha-1}|u|dx\\
&\leq
C\biggl(\int_{\R}|x|^{\alpha}\rho^{\beta}dx\biggl)^{\frac{1}{2}}\parallel|x|^{\frac{\alpha}{2}}u_{x}\parallel_{L^{2}}+C\biggl(\int_{\R}|x|^{\alpha}\rho^{\beta}dx\biggl)^{\frac{1}{2}}
\parallel|x|^{\frac{\alpha}{2}-1}u\parallel_{L^{2}}\\
&\leq
\varepsilon\parallel|x|^{\frac{\alpha}{2}}u_{x}\parallel_{L^{2}}^{2}+C\int_{\R}|x|^{\alpha}\rho^{\beta}dx.
\end{aligned}
\end{equation}
Substituting $(\ref{c27})-(\ref{c32})$ into\ $(\ref{c26})$ and
applying the Gronwall inequality lead to
$$\int_{\R}|x|^{\alpha}(\rho u^{2}+\rho^{\gamma}+\rho^{\beta})dx+\int_{0}^{T}\int_{\R}|x|^{\alpha}\mu(\rho)u_{x}^{2}dxdt\leq C(T).$$
The proof of the lemma is completed.
\end{proof}
\end{Lemma}

\subsection{Estimates of the higher order derivatives}
\begin{Lemma}\label{lemmac7}
Let $(\rho,u)$  be the smooth solution to $(\ref{a1})-(\ref{a4})$.
Then for any $T>0$, it holds
$$\int_{\R}\rho (u_{t})^{2}dx+\int_{0}^{T}\int_{\R}\mu(\rho)(u_{tx})^{2}dxdt+\int_{0}^{T}\int_{\R}\biggl(\rho_{t}^{2}+(\rho^{\gamma})_{t}^{2}
+(\rho^{\beta})_{t}^{2}\biggl)dxdt\leq C(T).$$
\begin{proof}
Differentiating $(\ref{a1})_{2}$ with respect to $t$, we have
\begin{equation}\label{c33}
\rho u_{tt}+\rho_{t}u_{t}+\rho_{t}uu_{x}+\rho u_{t}u_{x}+\rho uu_{xt}+(\rho^{\gamma})_{xt}=(\mu(\rho)u_{xt}+(\rho^{\beta})_{t}u_{x})_{x}.
\end{equation}
Multiplying on both sides of $(\ref{c33})$ by $u_{t}$, integrating
over $\R$ and using $(\ref{a1})_{1}$, we have
\begin{equation}\label{c34}
\begin{aligned}
\frac{1}{2}\frac{d}{dt}&\int_{\R}\rho u_{t}^{2}dx+\int_{\R}\mu(\rho)u_{tx}^{2}dx=-\int_{\R}\rho_{t}uu_{x}u_{t}dx-\int_{\R}\rho u_{t}^{2}u_{x}dx-\int_{\R}(\rho^{\gamma})_{tx}u_{t}dx\\
&-\int_{\R}(\rho^{\beta})_{t}u_{x}u_{xt}dx-2\int_{\R}\rho
uu_{t}u_{xt}dx\equiv K_{1}+K_{2}+K_{3}+K_{4}+K_{5}.
\end{aligned}
\end{equation}
Now we estimate the  terms $K_1-K_5$ as follows
\begin{equation}\label{c35}
\begin{aligned}
K_{1}&+K_{2}=-\int_{\R}\rho_{t}uu_{x}u_{t}dx-\int_{\R}\rho u_{t}^{2}u_{x}dx=\int_{\R}(\rho u)_{x}(uu_{x}u_{t})dx-\int_{\R}\rho u_{t}^{2}u_{x}dx\\
&=-\int_{\R}\rho u(u_{x}^{2}u_{t}+uu_{xx}u_{t}+uu_{x}u_{xt})dx-\int_{\R}\rho u_{t}^{2}u_{x}dx\\
&\leq
\int_{\R}(\sqrt{\rho}|u_{t}|)(\sqrt{\rho}|u|)|u_{x}|^{2}dx+\int_{\R}(\sqrt{\rho}|u_{t}|)(\sqrt{\rho}u^{2})u_{xx}dx+\int_{\R}(\sqrt{\rho}u)u_{xt}(\sqrt{\rho}uu_{x})dx
\\
&\ \ \ \ \ \ +\|u_{x}\|_{L^{\infty}}\int_{\R}\rho u_{t}^{2}dx\\
&\leq C\|u_{x}\|_{L^{\infty}}^{2}\|\sqrt{\rho}u_{t}\|_{L^{2}}\|\sqrt{\rho}u\|_{L^{2}}+C\|u\|_{L^{\infty}}^{2}\|\sqrt{\rho}u_{t}\|_{L^{2}}\|u_{xx}\|_{L^{2}}\\
&\ \ \ \ \ \ +C\|u\|_{L^{\infty}}
\|u_{x}\|_{L^{\infty}}\|u_{tx}\|_{L^{2}}\|\sqrt{\rho}u\|_{L^{2}}+C\|u_{x}\|_{L^{\infty}}\int_{\R}\rho u_{t}^{2}dx\\
&\leq C(\|\sqrt{\rho}u_{t}\|_{L^{2}}+1)^{2}\|\sqrt{\rho}u_{t}\|_{L^{2}}+C(\||x|^{\frac{\alpha}{2}}u_{x}\|_{L^{2}}+1)^{2}\|\sqrt{\rho}u_{t}\|_{L^{2}}(\|\sqrt{\rho}u_{t}\|_{L^{2}}+1)\\
&\ \ \ \ \ \ +C(\|\sqrt{\rho}u_{t}\|_{L^{2}}+1)(\||x|^{\frac{\alpha}{2}}u_{x}\|_{L^{2}}+1)\|u_{xt}\|_{L^{2}}+C(\|\sqrt{\rho}u_{t}\|_{L^{2}}+1)\int_{\R}\rho u_{t}^{2}dx\\
&\leq
\varepsilon\int_{\R}\mu(\rho)u_{xt}^{2}dx+C(\||x|^{\frac{\alpha}{2}}u_{x}\|_{L^{2}}^{2}+\int_{\R}\rho
u_{t}^{2}dx+1)\int_{\R}\rho
u_{t}^{2}dx+C(\||x|^{\frac{\alpha}{2}}u_{x}\|_{L^{2}}^{2}+1),
\end{aligned}
\end{equation}
where we used the fact
$$\|u_{x}\|_{L^{\infty}}+\|u_{xx}\|_{L^{2}}\leq C(\|\sqrt{\rho}u_{t}\|_{L^{2}}+1),$$
which follows from $(\ref{c18})$ and $(\ref{c23})$.\\ Note that
\begin{equation}\label{95-1}
\begin{aligned}
\|u\|_{L^{\infty}}&\leq \|u\|_{L^{\frac{2}{\alpha-1}}}^{\frac{1}{\alpha}}\|u_{x}\|_{L^{2}}^{1-\frac{1}{\alpha}}\leq C(\|u\|_{L^{\frac{2}{\alpha-1}}}+\|u_{x}\|_{L^{2}})\\
&\leq C(\||x|^{\frac{\alpha}{2}}u_{x}\|_{L^{2}}+1).
\end{aligned}
\end{equation}
Using  $(\ref{a1})_{1}$, we have
\begin{equation}\label{c36}
\begin{aligned}
\|\rho_{t}\|_{L^{2}}\leq C(\|\rho u_{x}\|_{L^{2}}+\|\rho_{x}u\|_{L^{2}})\leq C(1+\|u\|_{L^{\infty}})\leq C(1+\||x|^{\frac{\alpha}{2}}u_{x}\|_{L^{2}}).
\end{aligned}
\end{equation}
Thanks to Lemma \ref{lemmac6}, we have
$$\int_{0}^{T}\int_{\R}\biggl(\rho_{t}^{2}\biggl)dxdt\leq
C(T).$$\\
Similarly, we can obtain
$$\int_{0}^{T}\int_{\R}\biggl(\rho_{t}^{2}+(\rho^{\gamma})_{t}^{2}+(\rho^{\beta})_{t}^{2}\biggl)dxdt\leq
C(T).$$\\
Integration by parts implies
\begin{equation}\label{c37}
\begin{aligned}
K_{3}&=-\int_{\R}(\rho^{\gamma})_{tx}u_{t}dx=\int_{\R}(\rho^{\gamma})_{t}u_{tx}dx=\int_{\R}\gamma\rho^{\gamma-1}\rho_{t}u_{tx}dx\\
&\leq C\|\rho_{t}\|_{L^{2}}\|u_{tx}\|_{L^{2}}\leq
\varepsilon\|u_{tx}\|_{L^{2}}^{2}+C\|\rho_{t}\|_{L^{2}}^{2}.
\end{aligned}
\end{equation}
\begin{equation}\label{c38}
\begin{aligned}
K_{4}&=-\int_{\R}(\rho^{\beta})_{t}u_{x}u_{xt}dx\leq C\|u_{x}\|_{L^{\infty}}\|(\rho^{\beta})_{t}\|_{L^{2}}\|u_{xt}\|_{L^{2}}\\
&\leq \varepsilon\|u_{xt}\|_{L^{2}}^{2}+C\|u_{x}\|_{L^{\infty}}^{2}\|(\rho^{\beta})_{t}\|_{L^{2}}^{2}\\
&\leq \varepsilon\|u_{xt}\|_{L^{2}}^{2}+C(\|\sqrt{\rho}u_{t}\|_{L^{2}}^{2}+1)\|(\rho^{\beta})_{t}\|_{L^{2}}^{2}.
\end{aligned}
\end{equation}
$K_5$ is estimated as
\begin{equation}\label{c39}
\begin{aligned}
K_{5}&=-2\int_{\R}\rho uu_{t}u_{xt}dx\leq C\int_{\R}(\sqrt{\rho}u_{t})(\sqrt{\rho}u)u_{xt}dx\leq C\|u\|_{L^{\infty}}\|\sqrt{\rho}u_{t}\|_{L^{2}}\|u_{xt}\|_{L^{2}}\\
&\leq \varepsilon\|u_{xt}\|_{L^{2}}^{2}+C\|u\|_{L^{\infty}}^{2}\int_{\R}\rho u_{t}^{2}dx\\
&\leq
\varepsilon\|u_{xt}\|_{L^{2}}^{2}+C(\||x|^{\frac{\alpha}{2}}u_{x}\|_{L^{2}}^{2}+1)\int_{\R}\rho
u_{t}^{2}dx.
\end{aligned}
\end{equation}
Substituting \eqref{c35}-\eqref{c39}\ into\ \eqref{c34}, integrating
the resulted equation with respect to $t$ over $[0,T]$, using Lemmas
$\ref{lemmac3}$\ and\ $\ref{lemmac6}$ and the Gronwall inequality,
we obtain
$$\int_{\R}\rho u_{t}^{2}dx+\int_{0}^{T}\int_{\R}\mu(\rho)u_{tx}^{2}dxdt\leq C(T).$$
The proof of the lemma  is completed.
\end{proof}
\end{Lemma}

\begin{Lemma}\label{lemmac8}
Let $(\rho,u)$  be the smooth solution to $(\ref{a1})-(\ref{a4})$.
Then for any $T>0$, it holds
$$\|u_{x}\|_{L^{\infty}}(t)+\int_{\R}(u_{xx})^{2}dx\leq C(T).$$
\begin{proof}
Applying $(\ref{c23})$, \ Lemma\ \ref{lemmac4} and Lemma
\ref{lemmac7}, we have
\begin{equation}\nonumber
\|u_{xx}\|_{L^{2}}\leq
C\bigg[\|\sqrt{\rho}u_{t}\|_{L^{2}}+1+\|(\rho^{\gamma})_{x}\|_{L^{2}}
+(\|\sqrt{\rho}u_{t}\|_{L^{2}}+1)\|(\rho^{\beta})_{x}\|_{L^{2}}\bigg]\leq
C(T),
\end{equation}
and
\begin{equation}\nonumber
\|u_{x}\|_{L^{\infty}}\leq
C(\|u_{x}\|_{L^{2}}+\|u_{xx}\|_{L^{2}})\leq C(T).
\end{equation}
Then we complete the proof of the lemma.
\end{proof}
\end{Lemma}

Inspired by Hoff \cite{hoff95}, we estimate the weighted material
derivative $\dot{u}=(\partial_{t}+u\partial_{x})u$.
\begin{Lemma}\label{lemmac9}
Let $(\rho,u)$  be the smooth solution to $(\ref{a1})-(\ref{a4})$.
Then for any $T>0$ and
$2<\alpha<1+\frac{2}{\sqrt[3]{1+\sqrt[3]{4}}}$, it holds
$$\int_{\R}\rho\dot{u}^{2}|x|^{\alpha}dx+\int_{0}^{T}\int_{\R}(1+|x|^{\alpha})\mu(\rho)\dot{u}_{x}^{2}dxdt+\int_{0}^{T}\|u_{t}\|_{L^{\frac{2}{\alpha-1}}}^{2}dt\leq C(T).$$
\begin{proof}
Firstly, applying $\partial_{t}+\partial_{x}(u\cdot)$ to equation $(\ref{a1})_{2}$ gives
\begin{equation}\label{c41}
\begin{aligned}
\rho \dot{u}_{t}+\rho
u\dot{u}_{x}-(\mu(\rho)\dot{u}_{x})_{x}=(\gamma
pu_{x}-\mu(\rho)u_{x}^{2}-\beta\rho^{\beta}u_{x}^{2})_{x}.
\end{aligned}
\end{equation}
Multiplying on both sides of $(\ref{c41})$ by $|x|^{\alpha}\dot{u}$
and integrating over $\R$, we have
\begin{equation}\label{c42}
\begin{aligned}
&\frac{1}{2}\frac{d}{dt}\int_{\R}\rho \dot{u}^{2}|x|^{\alpha}dx+\int_{\R}\mu(\rho)\dot{u}_{x}^{2}|x|^{\alpha}dx\\
&=\frac{1}{2}\alpha\int_{\R}\rho u|x|^{\alpha-2}x\dot{u}^{2}dx-\int_{\R}\mu(\rho)\dot{u}_{x}\dot{u}\alpha|x|^{\alpha-2}xdx\\
&\ \ \ \ -\int_{\R}(\gamma pu_{x}-\mu(\rho)u_{x}^{2}-\beta\rho^{\beta}u_{x}^{2})(|x|^{\alpha}\dot{u}_{x}+\alpha|x|^{\alpha-2}x\dot{u})dx\\
&\equiv L_{1}+L_{2}+L_{3}.\\
\end{aligned}
\end{equation}
Direct estimates give
\begin{equation}\label{d1}
\begin{aligned}
&L_{1}\leq|\frac{1}{2}\alpha\int_{\R}\rho u|x|^{\alpha-2}x\dot{u}^{2}dx|\leq\frac{\alpha}{2}\int_{\R} \rho|u||x|^{\alpha-1}\dot{u}^{2}dx=\frac{\alpha}{2}\int_{\R}(\rho\dot{u}^{2}|x|^{\alpha})^{\frac{\alpha-1}{\alpha}}(\rho\dot{u}^{2})^{\frac{1}{\alpha}}|u|dx\\
&\leq C\|u\|_{L^{\infty}}\int_{\R} \rho\dot{u}^{2}|x|^{\alpha}dx+C\|u\|_{L^{\infty}}\int_{\R} \rho\dot{u}^{2}dx\\
&\leq C(\||x|^{\frac{\alpha}{2}}u_{x}\|_{L^{2}}+1)\int_{\R}\rho\dot{u}^{2}|x|^{\alpha}dx+C(\||x|^{\frac{\alpha}{2}}u_{x}\|_{L^{2}}+1)(\int_{\R} \rho u_{t}^{2}dx+\int_{\R}\rho u^{2}u_{x}^{2}dx)\\
&\leq C(\||x|^{\frac{\alpha}{2}}u_{x}\|_{L^{2}}+1)\int_{\R}
\rho\dot{u}^{2}|x|^{\alpha}dx+C(\||x|^{\frac{\alpha}{2}}u_{x}\|_{L^{2}}+1),
\end{aligned}
\end{equation}
\begin{equation}\label{d2}
\begin{aligned}
L_{2}&\leq |-\int_{\R}\mu(\rho)\dot{u}_{x}\dot{u}\alpha|x|^{\alpha-2}xdx|\leq|\int_{\R}\dot{u}_{x}\dot{u}\alpha|x|^{\alpha-2}xdx+\int_{\R}\rho^{\beta}\dot{u}_{x}\dot{u}\alpha|x|^{\alpha-2}xdx|\\
&\leq
\frac{\alpha(\alpha-1)}{2}\int_{\R}\dot{u}^{2}|x|^{\alpha-2}dx+C|\int_{\R}(|x|^{\frac{\alpha}{2}}\dot{u}_{x})(|x|^{\frac{\alpha}{3}-1}\dot{u})(\rho^{\beta}|x|^{\alpha})
^{\frac{1}{6}}\rho^{\frac{5\beta}{6}}dx|\\
&\leq \frac{\alpha(\alpha-1)^{3}}{8}\||x|^{\frac{\alpha}{2}}\dot{u}_{x}\|_{L^{2}}^{2}+C\||x|^{\frac{\alpha}{3}-1}\dot{u}\|_{L^{3}}\||x|^{\frac{\alpha}{2}}\dot{u}_{x}\|_{L^{2}}
\|\rho^{\beta}|x|^{\alpha}\|_{L^{1}}\\
&\leq \frac{\alpha(\alpha-1)^{3}}{8}\||x|^{\frac{\alpha}{2}}\dot{u}_{x}\|_{L^{2}}^{2}+C\|\dot{u}_{x}\|_{L^{2}}^{1-\theta}\||x|^{\frac{\alpha}{2}}\dot{u}_{x}\|_{L^{2}}^{1+\theta}\\
&\leq
\frac{\alpha(\alpha-1)^{3}}{8}\||x|^{\frac{\alpha}{2}}\dot{u}_{x}\|_{L^{2}}^{2}+\varepsilon\||x|^{\frac{\alpha}{2}}\dot{u}_{x}\|_{L^{2}}^{2}+C\|\dot{u}_{x}\|_{L^{2}}^{2}.
\end{aligned}
\end{equation}
The index $\alpha$ and $\theta$ in  $(\ref{d2})$ are chosen to
satisfy
$$\frac{1}{3}+\frac{\alpha}{3}-1=(\frac{1}{2}+\frac{-1}{1})(1-\theta)+(\frac{1}{2}+\frac{\frac{\alpha}{2}-1}{1})\theta,$$
which implies $$\theta=\frac{2\alpha-1}{3\alpha}<\frac{2}{3}.$$
Moreover, the restriction \eqref{a6+} guarantees that
 $$\frac{1}{2}+\frac{\frac{\alpha}{2}-1}{1}>0,\ \frac{1}{3}+\frac{\alpha}{3}-1>0, \ \frac{\alpha(\alpha-1)^{3}}{8}<1.$$
Concerning $L_3$, we have
\begin{equation}\label{d3}
\begin{aligned}
L_{3}&=-\int_{\R}(\gamma pu_{x}-\mu(\rho)u_{x}^{2}-\beta\rho^{\beta}u_{x}^{2})(|x|^{\alpha}\dot{u}_{x}+\alpha|x|^{\alpha-2}x\dot{u})dx\\
&\leq C\int_{\R}(\gamma p+\mu(\rho)|u_{x}|+\beta\rho^{\beta}|u_{x}|)(|x|^{\frac{\alpha}{2}}|u_{x}|)(|\dot{u}_{x}||x|^{\frac{\alpha}{2}})dx\\
&\ \ \ +C\int_{\R}(\gamma p+\mu(\rho)|u_{x}|+\beta\rho^{\beta}|u_{x}|)(|x|^{\frac{\alpha}{2}}|u_{x}|)(|\dot{u}||x|^{\frac{\alpha}{2}-1})dx\\
&\leq C\||x|^{\frac{\alpha}{2}}u_{x}\|_{L^{2}}\||x|^{\frac{\alpha}{2}}\dot{u}_{x}\|_{L^{2}}\\
&\leq \varepsilon\||x|^{\frac{\alpha}{2}}\dot{u}_{x}\|_{L^{2}}^{2}+C\||x|^{\frac{\alpha}{2}}u_{x}\|_{L^{2}}^{2}
\end{aligned}
\end{equation}
Since
\begin{equation}\nonumber
\begin{aligned}
\|\dot{u}_{x}\|_{L^{2}}&\leq C\|(u_{t}+uu_{x})_{x}\|_{L^{2}}\leq C\biggl(\|u_{xt}\|_{L^{2}}+\|u_{x}^{2}\|_{L^{2}}+\|uu_{xx}\|_{L^{2}}\biggl)\\
&\leq C\biggl(\|u_{xt}\|_{L^{2}}+\|u\|_{L^{\infty}}+1\biggl)\leq
C\biggl(\|u_{xt}\|_{L^{2}}+\||x|^{\frac{\alpha}{2}}u_{x}\|_{L^{2}}+1\biggl),
\end{aligned}
\end{equation}
by Lemma $\ref{lemmac6}$\ and \ Lemma $\ref{lemmac7}$, we get
\begin{equation}\label{c44}
\begin{aligned}
\int_{0}^{T}\int_{\R}\dot{u}_{x}^{2}dxdt\leq C(T).
\end{aligned}
\end{equation}
Substituting \eqref{d1}-\eqref{d3}\ into\ \eqref{c42} and using
$(\ref{c44})$, Lemma $\ref{lemmac6}$ and the Gronwall inequality, we
derive
$$\int_{\R}\rho\dot{u}^{2}|x|^{\alpha}dx+\int_{0}^{T}\int_{\R}\mu(\rho)\dot{u}_{x}^{2}|x|^{\alpha}dxdt\leq C(T).$$
It follows from  the Caffarelli-Kohn-Nirenberg weighted inequality
that
\begin{equation}\nonumber
\begin{aligned}
\int_{0}^{T}\|u_{t}\|_{L^{\frac{2}{\alpha-1}}}^{2}dt&\leq
\int_{0}^{T}\|\dot{u}-uu_{x}\|_{L^{\frac{2}{\alpha-1}}}^{2}dt\leq \int_{0}^{T}\|\dot{u}\|_{L^{\frac{2}{\alpha-1}}}^{2}dt+\int_{0}^{T}\|u_{x}\|_{L^{\infty}}^{2}\|u\|_{L^{\frac{2}{\alpha-1}}}^{2}dt\\
&\leq
\int_{0}^{T}\||x|^{\frac{\alpha}{2}}\dot{u}_{x}\|_{L^{2}}^{2}dt+C\int_{0}^{T}\||x|^{\frac{\alpha}{2}}{u}_{x}\|_{L^{2}}^{2}dt\leq
C(T).
\end{aligned}
\end{equation}
The proof of the lemma is completed.
\end{proof}
\end{Lemma}

\begin{Lemma}\label{lemmac10}
Let $(\rho,u)$  be the smooth solution to $(\ref{a1})-(\ref{a4})$.
Then for any $T>0$, it holds
$$\int_{\R}(\rho_{xx}^{2}+(\rho^{\gamma})_{xx}^{2}+(\rho^{\beta})_{xx}^{2})dx
+\int_{0}^{T}\int_{\R}\rho_{xt}^{2}dxdt\leq C(T).$$
\begin{proof}
Differentiating $(\ref{a1})_{1}$ with respect to $x$ twice, we get
\begin{equation}\label{c46}
\begin{aligned}
(\rho_{xx})_{t}+3\rho_{xx}u_{x}+3\rho_{x}u_{xx}+\rho u_{xxx}+u\rho_{xxx}=0.
\end{aligned}
\end{equation}
Multiplying $(\ref{c46})$ by $\rho_{xx}$, integrating over $\R$ and
using $(\ref{a1})_{1}$, we have
\begin{equation}\label{c47}
\begin{aligned}
\frac{1}{2}\frac{d}{dt}\int_{\R}\rho_{xx}^{2}dx&=-3\int_{\R}\rho_{xx}^{2}u_{x}dx-3\int_{\R}\rho_{x}u_{xx}\rho_{xx}dx-\int_{\R}\rho u_{xxx}\rho_{xx}dx-\int_{\R}u\rho_{xxx}\rho_{xx}dx\\
&\leq C\|u_{x}\|_{L^{\infty}}\int_{\R}\rho_{xx}^{2}dx+C\|\rho_{x}\|_{L^{\infty}}\|u_{xx}\|_{L^{2}}\|\rho_{xx}\|_{L^{2}}\\
&\ \ \ \ +C\|\rho_{xx}\|_{L^{2}}\|u_{xxx}\|_{L^{2}}
+C\|u_{x}\|_{L^{\infty}}\int_{\R}\rho_{xx}^{2}dx\\
&\leq C\int_{\R}\rho_{xx}^{2}dx+C(\|\rho_{xx}\|_{L^{2}}+1)\|\rho_{xx}\|_{L^{2}}+C\|u_{xxx}\|_{L^{2}}\|\rho_{xx}\|_{L^{2}}\\
&\leq C\int_{\R}\rho_{xx}^{2}dx+C(\|u_{xxx}\|_{L^{2}}^{2}+1).
\end{aligned}
\end{equation}
Similarly, we have
\begin{equation}\label{c48}
\begin{aligned}
\frac{d}{dt}\int_{\R}((\rho^{\gamma})_{xx}^{2}+(\rho^{\beta})_{xx}^{2})dx\leq
C\int_{\R}((\rho^{\gamma})_{xx}^{2}+(\rho^{\beta})_{xx}^{2})dx+C(\|u_{xxx}\|_{L^{2}}^{2}+1).
\end{aligned}
\end{equation}
Combining $(\ref{c47})$ with $(\ref{c48})$, we obtain
\begin{equation}\label{c49}
\begin{aligned}
\frac{d}{dt}\int_{\R}(\rho_{xx}^{2}+(\rho^{\gamma})_{xx}^{2}+(\rho^{\beta})_{xx}^{2})dx\leq
C\int_{\R}(\rho_{xx}^{2}+(\rho^{\gamma})_{xx}^{2}+(\rho^{\beta})_{xx}^{2})dx+C(\|u_{xxx}\|_{L^{2}}^{2}+1).
\end{aligned}
\end{equation}
Differentiating $(\ref{a1})_{2}$ with respect to $x$, we have
\begin{equation}\label{c50}
\begin{aligned}
\mu(\rho)u_{xxx}=-2(\rho^{\beta})_{x}u_{xx}-(\rho^{\beta})_{xx}u_{x}+\rho_{x}u_{t}+\rho
u_{xt}+\rho_{x}uu_{x}+\rho u_{x}^{2}+\rho
uu_{xx}+(\rho^{\gamma})_{xx},
\end{aligned}
\end{equation}
which yields
\begin{equation}\label{c51}
\begin{aligned}
\|u_{xxx}\|_{L^{2}}&\leq C\biggl(\|(\rho^{\beta})_{x}u_{xx}\|_{L^{2}}+\|(\rho^{\beta})_{xx}u_{x}\|_{L^{2}}+\|\rho_{x}u_{t}\|_{L^{2}}+\|\rho u_{xt}\|_{L^{2}}\\
&\ \ \ \ +\|\rho_{x}uu_{x}\|_{L^{2}}+\|\rho u_{x}^{2}\|_{L^{2}}+\|\rho uu_{xx}\|_{L^{2}}+\|(\rho^{\gamma})_{xx}\|_{L^{2}}\biggl)\\
&\leq \biggl(\|(\rho^{\beta})_{x}\|_{L^{\infty}}\|u_{xx}\|_{L^{2}}+\|u_{x}\|_{L^{\infty}}\|(\rho^{\beta})_{xx}\|_{L^{2}}+\|u_{t}\|_{L^{\infty}}\|\rho_{x}\|_{L^{2}}\\
&\ \ \ \ +\|\rho\|_{L^{\infty}}\|u_{xt}\|_{L^{2}}+\|u\|_{L^{\infty}}\|u_{x}\|_{L^{\infty}}\|\rho_{x}\|_{L^{2}}+\|\rho\|_{L^{\infty}}\|u_{x}\|_{L^{4}}^{2}\\
&\ \ \ \ \ \ \ +\|\rho u\|_{L^{\infty}}\|u_{xx}\|_{L^{2}}+\|(\rho^{\gamma})_{xx}\|_{L^{2}}\biggl)\\
&\leq C\biggl(\|\rho_{xx}\|_{L^{2}}+\|(\rho^{\beta})_{xx}\|_{L^{2}}+\||x|^{\frac{\alpha}{2}}\dot{u}_{x}\|_{L^{2}}+\|u_{xt}\|_{L^{2}}+\||x|^{\frac{\alpha}{2}}u_{x}\|_{L^{2}}\\
&\ \ \ \ \ \ \ +1+\|(\rho^{\gamma})_{xx}\|_{L^{2}}\biggl).
\end{aligned}
\end{equation}
Combining $(\ref{c49})$ with $(\ref{c51})$, we get
\begin{equation}\nonumber
\begin{aligned}
\frac{d}{dt}&\int_{\R}(\rho_{xx}^{2}+(\rho^{\gamma})_{xx}^{2}+(\rho^{\beta})_{xx}^{2})dx\\
&\leq
C\int_{\R}(\rho_{xx}^{2}+(\rho^{\gamma})_{xx}^{2}+(\rho^{\beta})_{xx}^{2})dx+C(\||x|^{\frac{\alpha}{2}}\dot{u}_{x}\|_{L^{2}}^{2}+\|u_{xt}\|_{L^{2}}^{2}+\||x|^{\frac{\alpha}{2}}u_{x}\|_{L^{2}}^{2}+1)
\end{aligned}
\end{equation}
Using Lemmas $\ref{lemmac6},\ \ref{lemmac7},\ \ref{lemmac9},$ and
the Gronwall inequality, we get
$$\int_{\R}(\rho_{xx}^{2}+(\rho^{\gamma})_{xx}^{2}+(\rho^{\beta})_{xx}^{2}dx\leq C(T).$$
Furthermore, differentiating $(\ref{a1})_{1}$ with respect to $x$, we have
\begin{equation}\nonumber
\begin{aligned}
\rho_{xt}+2\rho_{x}u_{x}+\rho u_{xx}+u\rho_{xx}=0
\end{aligned}
\end{equation}
which implies
\begin{equation}\label{c53}
\begin{aligned}
\|\rho_{xt}\|_{L^{2}}&\leq C(\|\rho_{x}u_{x}\|_{L^{2}}+\|\rho u_{xx}\|_{L^{2}}+\|u\rho_{xx}\|_{L^{2}})\\
&\leq C(1+\|u\|_{L^{\infty}})\\
&\leq C(1+\||x|^{\frac{\alpha}{2}}u_{x}\|_{L^{2}})
\end{aligned}
\end{equation}
and
$$\int_{0}^{T}\int_{\R}\rho_{xt}^{2}dxdt\leq C(T).$$
The proof of the lemma  is completed.
\end{proof}
\end{Lemma}

\begin{Lemma}\label{lemmac11}
Let $(\rho,u)$  be the smooth solution to $(\ref{a1})-(\ref{a4})$.
Then for any $T>0$\ and \
$2<\alpha<1+\frac{2}{\sqrt[3]{1+\sqrt[3]{4}}},$ it holds
$$\int_{\R}\bigg(\rho_{t}^{2}+(\rho^{\gamma})_{t}^{2}+(\rho^{\beta})_{t}^{2}+\rho_{xt}^{2}+(\rho^{\gamma})_{xt}^{2}+(\rho^{\beta})_{xt}^{2}+\mu(\rho)|x|^{\alpha}u_{x}^{2}\bigg)dx+\int_{0}^{T}\int_{\R}\rho \dot{u}^{2}|x|^{\alpha}dxdt\leq C(T).$$
\begin{proof}
Rewriting the equation $(\ref{a1})_{2}$ as
\begin{equation}\label{9-27}
\rho\dot{u}+(\rho^{\gamma})_{x}=(\mu(\rho)u_{x})_{x}.
\end{equation}
 Multiplying on both sides of \eqref{9-27} by
$|x|^{\alpha}\dot{u}$ and integrating on $\R$ with respect to $x$,
we have
\begin{equation}\label{c54}
\begin{aligned}
\frac{1}{2}\frac{d}{dt}\int_{\R}&\mu(\rho)|x|^{\alpha}u_{x}^{2}dx+\int_{\R}\rho\dot{u}^{2}|x|^{\alpha}dx\\
&=\frac{1}{2}\int_{\R}(\rho^{\beta})_{t}|x|^{\alpha}u_{x}^{2}dx-\int_{\R}\mu(\rho)u_{x}|x|^{\alpha}(uu_{x})_{x}dx\\
&\ \ \ \ -\int_{\R}\mu(\rho)u_{x}\alpha|x|^{\alpha-2}x\dot{u}dx
-\int_{\R}(\rho^{\gamma})_{x}|x|^{\alpha}\dot{u}dx\\
&\equiv M_{1}+M_{2}+M_{3}+M_{4}.
\end{aligned}
\end{equation}
Next we  estimate the terms $M_{1}-M_{4}$,
\begin{equation}\label{c55}
\begin{aligned}
M_{1}=\frac{1}{2}\int_{\R}(\rho^{\beta})_{t}|x|^{\alpha}u_{x}^{2}dx\leq
C\|(\rho^{\beta})_{t}\|_{L^{\infty}}\int_{\R}|x|^{\alpha}u_{x}^{2}dx\leq
C(1+\||x|^{\frac{\alpha}{2}}u_{x}\|_{L^{2}})\int_{\R}|x|^{\alpha}u_{x}^{2}dx,
\end{aligned}
\end{equation}
where
\begin{equation}\label{c56}
\begin{aligned}
\|\rho_{t}\|_{L^{\infty}}&\leq C(\|\rho_{t}\|_{L^{2}}+\|\rho_{xt}\|_{L^{2}})\\
&\leq C(\|\rho u_{x}\|_{L^{2}}+\|u\rho_{x}\|_{L^{2}}+\|\rho_{x}u_{x}\|_{L^{2}}+\|\rho u_{xx}\|_{L^{2}}+\|u\rho_{xx}\|_{L^{2}})\\
&\leq C(1+\|u\|_{L^{\infty}})\leq C(1+\||x|^{\frac{\alpha}{2}}u_{x}\|_{L^{2}}).
\end{aligned}
\end{equation}
Similar to $(\ref{c56})$, we  get
\begin{equation}\label{c57}
\begin{aligned}
\|(\rho^{\gamma})_{t}\|_{L^{\infty}}+\|(\rho^{\beta})_{t}\|_{L^{\infty}}\leq C(1+\||x|^{\frac{\alpha}{2}}u_{x}\|_{L^{2}}).
\end{aligned}
\end{equation}
Integration by parts yields
\begin{equation}\label{c58}
\begin{aligned}
M_{2}&=-\int_{\R}u_{x}|x|^{\alpha}(uu_{x})_{x}dx-\int_{\R}\rho^{\beta}u_{x}|x|^{\alpha}(uu_{x})_{x}dx\\
&=\frac{1}{2}\int_{\R}\alpha|x|^{\alpha-2}xuu_{x}^{2}dx-\frac{1}{2}\int_{\R}|x|^{\alpha}u_{x}^{3}dx
-\frac{1}{2}\int_{\R}\rho^{\beta}|x|^{\alpha}u_{x}^{3}dx\\
&\ \ \ \
+\frac{1}{2}\int_{\R}(\rho^{\beta})_{x}u_{x}^{2}|x|^{\alpha}udx+\frac{1}{2}\int_{\R}\rho^{\beta}
u_{x}^{2}u\alpha|x|^{\alpha-2}xdx\\
&\leq C\|u_{x}\|_{L^{\infty}}\||x|^{\frac{\alpha}{2}}u_{x}\|_{L^{2}}\||x|^{\frac{\alpha}{2}-1}u\|_{L^{2}}+C\|u_{x}\|_{L^{\infty}}\int_{\R}|x|^{\alpha}u_{x}^{2}dx\\
&\ \ \ \
+\|(\rho^{\beta})_{x}\|_{L^{\infty}}\|u\|_{L^{\infty}}\int_{\R}|x|^{\alpha}u_{x}^{2}dx
+C\|u_{x}\|_{L^{\infty}}
\||x|^{\frac{\alpha}{2}}u_{x}\|_{L^{2}}\||x|^{\frac{\alpha}{2}-1}u\|_{L^{2}}\\
&\leq
C(1+\||x|^{\frac{\alpha}{2}}u_{x}\|_{L^{2}})\int_{\R}|x|^{\alpha}u_{x}^{2}dx,
\end{aligned}
\end{equation}
\begin{equation}\label{c64}
\begin{aligned}
M_{3}&=-\int_{\R}\mu(\rho)u_{x}\alpha|x|^{\alpha-2}x\dot{u}\leq C\int_{\R}|u_{x}||\dot{u}||x|^{\alpha-1}=\int_{\R}(|x|^{\frac{\alpha}{2}}|u_{x}|)(|x|^{\frac{\alpha}{2}-1}|\dot{u}|)\\
&\leq C\||x|^{\frac{\alpha}{2}}u_{x}\|_{L^{2}}\||x|^{\frac{\alpha}{2}-1}\dot{u}\|_{L^{2}}\leq C\||x|^{\frac{\alpha}{2}}u_{x}\|_{L^{2}}\||x|^{\frac{\alpha}{2}}\dot{u}_{x}\|_{L^{2}}\\
&\leq
C\||x|^{\frac{\alpha}{2}}u_{x}\|_{L^{2}}^{2}+C\||x|^{\frac{\alpha}{2}}\dot{u}_{x}\|_{L^{2}}^{2}.
\end{aligned}
\end{equation}
Similarly, we have
\begin{equation}\label{c65}
\begin{aligned}
M_{4}&=-\int_{\R}(\rho^{\gamma})_{x}|x|^{\alpha}\dot{u}dx=\int_{\R}\rho^{\gamma}
(|x|^{\alpha}u_{x}+\alpha|x|^{\alpha-2}x\dot{u})dx\\
&\leq
C\|\rho^{\frac{\gamma}{2}}|x|^{\frac{\alpha}{2}}\|_{L^{2}}\||x|^{\frac{\alpha}{2}}\dot{u}_{x}\|_{L^{2}}
+C\|\rho^{\frac{\gamma}{2}}|x|^{\frac{\alpha}{2}}\|_{L^{2}}\||x|^{\frac{\alpha}{2}-1}\dot{u}\|_{L^{2}}\\
&\leq C\||x|^{\frac{\alpha}{2}}\dot{u}_{x}\|_{L^{2}}.
\end{aligned}
\end{equation}
Substituting \eqref{c55}, \eqref{c58}, \eqref{c64}, \eqref{c65}\
into\ \eqref{c54} yields
$$\int_{\R}\mu(\rho)|x|^{\alpha}u_{x}^{2}dx+\int_{0}^{T}\int_{\R}\rho \dot{u}^{2}|x|^{\alpha}dxdt\leq C(T).$$
It follows from $(\ref{c36})$\ and\ $(\ref{c53})$ that
$$\|\rho_{t}\|_{L^{2}}+\|\rho_{xt}\|_{L^{2}}\leq C(1+\||x|^{\frac{\alpha}{2}}u_{x}\|_{L^{2}})\leq C(T).$$
Similarly, we  have
$$\|(\rho^{\gamma})_{t}\|_{L^{2}}+\|(\rho^{\beta})_{t}\|_{L^{2}}+\|(\rho^{\gamma})_{xt}\|_{L^{2}}+\|(\rho^{\beta})_{xt}\|_{L^{2}}\leq C(T).$$
The proof of the lemma  is completed.
\end{proof}
\end{Lemma}

\begin{Lemma}\label{lemmac12}
Let $(\rho,u)$  be the smooth solution to $(\ref{a1})-(\ref{a4})$.
Then for any $0\leq t\leq T$, it holds
$$t\int_{\R}\mu(\rho)\dot{u}_{x}^{2}dx+\int_{0}^{T}t(\|\sqrt{\rho}\dot{u}_{t}\|_{L^{2}}^{2}+\|\dot{u}_{xx}\|_{L^{2}}^{2})dt\leq C(T).$$
\begin{proof}
Multiplying on both sides of $(\ref{c41})$  by $\dot{u}_{t}$ and
integrating the resulted equation with respect to $x$ over $\R$, we
have
\begin{equation}\label{c66}
\begin{aligned}
\frac{1}{2}\frac{d}{dt}&\int_{\R}\mu(\rho)\dot{u}_{x}^{2}dx+\int_{\R}\rho
\dot{u}_{t}^{2}dx=\frac{d}{dt}\int_{\R}(\mu(\rho)u_{x}^{2}+\beta\rho^{\beta}u_{x}^{2}-\gamma\rho^{\gamma}u_{x})\dot{u}_{x}dx\\
&+\frac{1}{2}\int_{\R}(\rho^{\beta})_{t}\dot{u}_{x}^{2}dx
-\int_{\R}\rho u\dot{u}_{x}\dot{u}_{t}dx-\int_{\R}((\rho^{\beta})_{t}u_{x}^{2}+\beta(\rho^{\beta})_{t}u_{x}^{2}-\gamma(\rho^{\gamma})_{t}u_{x})\dot{u}_{x}dx\\
&-\int_{\R}(2\mu(\rho)u_{x}+2\beta\rho^{\beta}
u_{x}-\gamma\rho^{\gamma})(\dot{u}-uu_{x})_{x}\dot{u}_{x}dx\\
\end{aligned}
\end{equation}
It follows that:
\begin{equation}\label{c67}
\begin{aligned}
\frac{1}{2}\int_{\R}&(\rho^{\beta})_{t}\dot{u}_{x}^{2}dx
-\int_{\R}\rho u\dot{u}_{x}\dot{u}_{t}dx\\
&\leq C\|(\rho^{\beta})_{t}\|_{L^{\infty}}\int_{\R}\dot{u}_{x}^{2}dx+\|\sqrt{\rho}\|_{L^{\infty}}\|u\|_{L^{\infty}}\|\sqrt{\rho}\dot{u}_{t}\|_{L^{2}}\|\dot{u}_{x}\|_{L^{2}}\\
&\leq C(1+\||x|^{\frac{\alpha}{2}}u_{x}\|_{L^{2}})\int_{\R}\dot{u}_{x}^{2}dx+C(1+\||x|^{\frac{\alpha}{2}}u_{x}\|_{L^{2}})\|\sqrt{\rho}\dot{u}_{t}\|_{L^{2}}\|\dot{u}_{x}\|_{L^{2}}\\
&\leq \varepsilon\int_{\R}\rho
\dot{u}_{t}^{2}dx+C\int_{\R}\dot{u}_{x}^{2}dx,
\end{aligned}
\end{equation}
where $(\ref{c57})$ has been used. Due to Lemma $\ref{lemmac11}$, we
obtain
\begin{equation}\label{c68}
\begin{aligned}
-\int_{\R}&((\rho^{\beta})_{t}u_{x}^{2}+\beta(\rho^{\beta})_{t}u_{x}^{2}-\gamma(\rho^{\gamma})_{t}u_{x})\dot{u}_{x}dx\\
&\leq C\|(\rho^{\beta})_{t}\|_{L^{\infty}}\|u_{x}\|_{L^{4}}^{2}\|\dot{u}_{x}\|_{L^{2}}+\|(\rho^{\gamma})_{t}\|_{L^{\infty}}\|u_{x}\|_{L^{2}}\|\dot{u}_{x}\|_{L^{2}}\\
&\leq C(1+\||x|^{\frac{\alpha}{2}}u_{x}\|_{L^{2}})\|\dot{u}_{x}\|_{L^{2}}\\
&\leq C\|\dot{u}_{x}\|_{L^{2}}.
\end{aligned}
\end{equation}
Similarly,
\begin{equation}\label{c69}
\begin{aligned}
-\int_{\R}&(2\mu(\rho)u_{x}+2\beta\rho^{\beta}
u_{x}-\gamma\rho^{\gamma})(\dot{u}-uu_{x})_{x}\dot{u}_{x}dx\\
&=-\int_{\R}(2\mu(\rho)u_{x}+2\beta\rho^{\beta}
u_{x}-\gamma\rho^{\gamma})(\dot{u}_{x}^{2}-u_{x}^{2}\dot{u}_{x}-uu_{xx}\dot{u}_{x})dx\\
&\leq C\int_{\R}\dot{u}_{x}^{2}dx+C\|u_{x}\|_{L^{4}}^{2}\|\dot{u}_{x}\|_{L^{2}}+C\|u\|_{L^{\infty}}\|u_{xx}\|_{L^{2}}\|\dot{u}_{x}\|_{L^{2}}\\
&\leq C\int_{\R}\dot{u}_{x}^{2}dx+C(T).
\end{aligned}
\end{equation}
Substituting $(\ref{c67}),(\ref{c68}),(\ref{c69})$\ into\
$(\ref{c66})$, multiplying the resulted equation by $t$ and
integrating respect to $t$ over $[\tau,t_{1}]$ with $\tau,\
t_{1}\in[0,T]$ give that
\begin{equation}\label{t1}
\begin{aligned}
&t_{1}\|\sqrt{\mu(\rho)}\dot{u}_{x}\|_{L^{2}}^{2}(t_{1})+\int_{\tau}^{t_{1}}t\|\sqrt{\rho}
\dot{u}_{t}\|_{L^{2}}^{2}dt\\
&\leq
\tau\|\sqrt{\mu(\rho)}\dot{u}_{x}\|_{L^{2}}^{2}(\tau)+t_{1}G(t_{1})-\tau
G(\tau)+\int_{\tau}^{t_{1}}\|\sqrt{\mu(\rho)}\dot{u}_{x}\|_{L^{2}}^{2}dt\\
&\ \ \ \ -\int_{\tau}^{t_{1}}
G(t)dt+C\int_{\tau}^{t_{1}}t\|\dot{u}_{x}\|_{L^{2}}^{2}dt\\
&\leq \tau\|\sqrt{\mu(\rho)}\dot{u}_{x}\|_{L^{2}}^{2}(\tau)-\tau
G(\tau)+C\int_{\tau}^{t_{1}}t\|\dot{u}_{x}\|_{L^{2}}^{2}dt+C(T)
\end{aligned}
\end{equation}
where
$$G(t)=\int_{\R}(\mu(\rho)u_{x}^{2}+\beta\rho^{\beta}u_{x}^{2}-\gamma\rho^{\gamma}u_{x})\dot{u}_{x}dx.$$
It follows from Lemma \ref{lemmac9} that $G(t)\in L^{1}(0,T),\
\|\sqrt{\mu(\rho)}\dot{u}_{x}\|_{L^{2}}\in L^{2}(0,T)$. Thus, there
exists a subsequence $\tau_{k}$ such that
$$\tau_{k}\rightarrow 0,\ \ \tau_{k}G(\tau_{k})\rightarrow 0,\ \ \tau_{k}\|\sqrt{\mu(\rho)}\dot{u}_{x}\|_{L^{2}}^{2}(\tau_{k})\rightarrow0,\ \ as\
\ k\rightarrow+\infty.$$  Taking $\tau=\tau_{k}$ in (\ref{t1}), then
letting $k\rightarrow+\infty$ and using the Gronwall inequality, one
gets that
$$t\|\sqrt{\mu(\rho)}\dot{u}_{x}\|_{L^{2}}^{2}+\int_{0}^{T}t\|\sqrt{\rho}\dot{u}_{t}\|_{L^{2}}^{2}dt
\leq C(T).$$ It follows from \eqref{c41} that
$$\mu(\rho)\dot{u}_{xx}=\rho\dot{u}_{t}+\rho u\dot{u}_{x}-(\gamma
pu_{x}-\mu(\rho)u_{x}^{2}-\beta\rho^{\beta}u_{x}^{2})_{x}-(\rho^{\beta})_{x}\dot{u}_{x}$$
Which yields
\begin{equation}\nonumber
\begin{aligned}
\|\dot{u}_{xx}\|_{L^{2}}&\leq
C\bigg(\|\rho\dot{u}_{t}\|_{L^{2}}+\|\rho
u\dot{u}_{x}\|_{L^{2}}+\|(\gamma
pu_{x}-\mu(\rho)u_{x}^{2}-\beta\rho^{\beta}u_{x}^{2})_{x}\|_{L^{2}}+\|(\rho^{\beta})_{x}\dot{u}_{x}\|_{L^{2}}\bigg)\\
&\leq
C(\|\sqrt{\rho}\dot{u}_{t}\|_{L^{2}}+\|\dot{u}_{x}\|_{L^{2}}+1)
\end{aligned}
\end{equation}
and
$$\int_{0}^{T}t\|\dot{u}_{xx}\|_{L^{2}}^{2}dt\leq C(T)$$
The proof of the lemma is completed.
\end{proof}
\end{Lemma}

\begin{Lemma}\label{lemmac13}
Let $(\rho,u)$  be the smooth solution to $(\ref{a1})-(\ref{a4})$.
Then for any $0\leq t\leq T$, it holds
$$t^{2}\int_{\R}\rho\dot{u}_{t}^{2}dx+
\int_{0}^{T}t^{2}\int_{\R}\mu(\rho)\dot{u}_{xt}^{2}dxdt\leq C(T).$$
\begin{proof}
Differentiating $(\ref{c41})$ with respect to $t$, we have
\begin{equation}\label{c70}
\begin{aligned}
&\rho\dot{u}_{tt}+\rho_{t}\dot{u}_{t}+\rho_{t}u\dot{u}_{x}+\rho u_{t}\dot{u}_{x}+\rho u\dot{u}_{xt}-(\mu(\rho)\dot{u}_{xt})_{x}-((\rho^{\beta})_{t}\dot{u}_{x})_{x}\\
&=(\gamma(\rho^{\gamma})_{t}u_{x}+\gamma\rho^{\gamma}u_{xt}-(\rho^{\beta})_{t}u_{x}^{2}-2\mu(\rho)u_{x}u_{xt}-\beta(\rho^{\beta})_{t}u_{x}^{2}-2\beta\rho^{\beta}
u_{x}u_{xt})_{x}.
\end{aligned}
\end{equation}
Multiplying on both sides of $(\ref{c70})$ by $\dot{u}_{t}$,
integrating the resulted equation  over $\R$, we have
\begin{equation}\label{c71}
\begin{aligned}
\frac{1}{2}&\frac{d}{dt}\int_{\R}\rho \dot{u}_{t}^{2}dx+\int_{\R}\mu(\rho)\dot{u}_{xt}^{2}dx\\
&\leq -\int_{\R}\rho_{t}u\dot{u}_{x}\dot{u}_{t}dx-\int_{\R}\rho u_{t}\dot{u}_{x}\dot{u}_{t}dx-\int_{\R}(\rho^{\beta})_{t}\dot{u}_{x}\dot{u}_{xt}dx-2\int_{\R}\rho u\dot{u}_{xt}\dot{u}_{t}dx\\
&\ \ \ \ +C\|\dot{u}_{xt}\|_{L^{2}}(\|(\rho^{\gamma})_{t}\|_{L^{\infty}}\|u_{x}\|_{L^{2}}+\|u_{xt}\|_{L^{2}}+\|(\rho^{\beta})_{t}\|_{L^{\infty}}\|u_{x}\|_{L^{4}}^{2})\\
&\leq -\int_{\R}\rho_{t}u\dot{u}_{x}\dot{u}_{t}dx-\int_{\R}\rho u_{t}\dot{u}_{x}\dot{u}_{t}dx-\int_{\R}(\rho^{\beta})_{t}\dot{u}_{x}\dot{u}_{xt}dx-2\int_{\R}\rho u\dot{u}_{xt}\dot{u}_{t}dx+C\|\dot{u}_{xt}\|_{L^{2}}\\
\end{aligned}
\end{equation}
It follows  that
\begin{equation}\label{c72}
\begin{aligned}
-\int_{\R}\rho_{t}u\dot{u}_{x}\dot{u}_{t}dx&=\int_{\R}(\rho
u)_{x}u\dot{u}_{x}\dot{u}_{t}dx\\
&=-\int_{\R}\rho
u(u_{x}\dot{u}_{x}\dot{u}_{t}+u\dot{u}_{xx}\dot{u}_{t}+u\dot{u}_{x}\dot{u}_{tx})dx\\
&\leq
C(\|\sqrt{\rho}\dot{u}_{t}\|_{L^{2}}\|\dot{u}_{x}\|_{L^{2}}+\|\sqrt{\rho}\dot{u}_{t}\|_{L^{2}}
\|\dot{u}_{xx}\|_{L^{2}}+\|\dot{u}_{x}\|_{L^{2}}\|\dot{u}_{tx}\|_{L^{2}})\\
&\leq \varepsilon
\|\dot{u}_{tx}\|_{L^{2}}^{2}+C(\|\sqrt{\rho}\dot{u}_{t}\|_{L^{2}}^{2}+\|\dot{u}_{x}\|_{L^{2}}^{2}
+\|\dot{u}_{xx}\|_{L^{2}}^{2})
\end{aligned}
\end{equation}
\begin{equation}\label{c73}
\begin{aligned}
-\int_{\R}\rho u_{t}\dot{u}_{x}\dot{u}_{t}dx&\leq C\int_{\R}\sqrt{\rho}|\dot{u}_{t}||\sqrt{\rho}u_{t}||\dot{u}_{x}|dx\\
&\leq C\|u_{t}\|_{L^{\infty}}\|\sqrt{\rho}\dot{u}_{t}\|_{L^{2}}\|\dot{u}_{x}\|_{L^{2}}\\
&\leq C(\||x|^{\frac{\alpha}{2}}\dot{u}_{x}\|_{L^{2}}+1)\|\sqrt{\rho}\dot{u}_{t}\|_{L^{2}}\\
&\leq C\int_{\R}\rho
\dot{u}_{t}^{2}dx+C(\||x|^{\frac{\alpha}{2}}\dot{u}_{x}\|_{L^{2}}^{2}+1).
\end{aligned}
\end{equation}
Similarly, we have
\begin{equation}\label{c74}
\begin{aligned}
-\int_{\R}(\rho^{\beta})_{t}\dot{u}_{x}\dot{u}_{xt}dx&-2\int_{\R}\rho u\dot{u}_{xt}\dot{u}_{t}dx\\
&\leq C\|(\rho^{\beta})_{t}\|_{L^{\infty}}\|\dot{u}_{x}\|_{L^{2}}\|\dot{u}_{xt}\|_{L^{2}}+C\|u\|_{L^{\infty}}\|\dot{u}_{xt}\|_{L^{2}}\|\sqrt{\rho}\dot{u}_{t}\|_{L^{2}}\\
&\leq C\|\dot{u}_{xt}\|_{L^{2}}(1+\|\sqrt{\rho}\dot{u}_{t}\|_{L^{2}})\\
&\leq \varepsilon\|\dot{u}_{xt}\|_{L^{2}}^{2}+C\int_{\R}\rho
\dot{u}_{t}^{2}dx+C(T),
\end{aligned}
\end{equation}
\begin{equation}\label{c74-1}
\begin{aligned}
C\|\dot{u}_{xt}\|_{L^{2}}\leq
\varepsilon\|\dot{u}_{xt}\|_{L^{2}}^{2}+C(T).
\end{aligned}
\end{equation}
Substituting (\ref{c72}), (\ref{c73}), (\ref{c74}), \eqref{c74-1}\
into\ (\ref{c71}), multiplying the resulted equation by $t^{2}$ and
integrating respect to $t$ over $[\tau,t_{1}]$ with $\tau,\
t_{1}\in[0,T]$ give that
\begin{equation}\label{t2}
\begin{aligned}
t_{1}^{2}\|\sqrt{\rho}\dot{u}_{t}\|_{L^{2}}^{2}(t_{1})&+\int_{\tau}^{t_{1}}t^{2}\|\sqrt
{\mu(\rho)}\dot{u}_{xt}\|_{L^{2}}^{2}dt\\
&\leq
\tau^{2}\|\sqrt{\rho}\dot{u}_{t}\|_{L^{2}}^{2}+2\int_{\tau}^{t_{1}}t\|\sqrt{\rho}\dot{u}_{t}\|_{L^{2}}^{2}dt+C\int_{\tau}^{t_{1}}
t^{2}\|\sqrt{\rho}\dot{u}_{t}\|_{L^{2}}^{2}dt\\
&\ \ \ \
+C\int_{\tau}^{t_{1}}t^{2}(\|\dot{u}_{x}\|_{L^{2}}^{2}+\|\dot{u}_{xx}\|_{L^{2}}^{2}
+\||x|^{\frac{\alpha}{2}}\dot{u}_{x}\|_{L^{2}}^{2}+1)dt.
\end{aligned}
\end{equation}
Since $$t\|\sqrt{\rho}\dot{u}_{t}\|_{L^{2}}\in L^{2}(0,T).$$ Thus
there exists a subsequence $\tau_{k}$ such that
$$\tau_{k}\rightarrow0,\ \
\tau_{k}^{2}\|\sqrt{\rho}\dot{u}_{t}(\tau_{k})\|_{L^{2}}^{2}\rightarrow0,\
\ as\ \ k\rightarrow +\infty.$$ Taking $\tau=\tau_{k}$ in
(\ref{t2}), then letting $k\rightarrow +\infty$ and using the
Gronwall inequality, one gets that
$$t^{2}\int_{\R}\rho\dot{u}_{t}^{2}dx+\int_{0}^{T}t^{2}\int_{\R}\mu(\rho)\dot{u}_{xt}^{2}dxdt\leq
C(T).$$
The proof of the lemma is completed.
\end{proof}
\end{Lemma}

\begin{Lemma}\label{lemmac14}
Let $(\rho,u)$  be the smooth solution to $(\ref{a1})-(\ref{a4})$.
Then for any $0\leq t\leq T, \
2<\alpha<1+\frac{2}{\sqrt[3]{1+\sqrt[3]{4}}}$, it holds
$$t\|\sqrt{\mu(\rho)}|x|^{\frac{\alpha}{4}}\dot{u}_{x}\|_{L^{2}}^{2}+\int_{0}^{T}t\|\sqrt{\rho}\dot{u}_{t}|x|^{\frac{\alpha}{4}}\|_{L^{2}}^{2}dt\leq C(T).$$
\begin{proof}
Multiplying on both sides of $(\ref{c41})$ by
$|x|^{\frac{\alpha}{2}}\dot{u}_{t}$, integrating the resulted
equation over $\R$, we have
\begin{equation}\label{c75}
\begin{aligned}
\frac{1}{2}&\frac{d}{dt}\int_{\R}\mu(\rho)\dot{u}_{x}^{2}|x|^{\frac{\alpha}{2}}dx+\int_{\R}\rho \dot{u}_{t}^{2}|x|^{\frac{\alpha}{2}}dx\\
&=\frac{d}{dt}\int_{\R}(\mu(\rho)u_{x}^{2}+\beta\rho^{\beta}u_{x}^{2}-\gamma\rho^{\gamma}u_{x})(|x|^{\frac{\alpha}{2}}\dot{u}_{x}
+\frac{\alpha}{2}|x|^{\frac{\alpha}{2}-2}x\dot{u})dx\\
&\ \
+\frac{1}{2}\int_{\R}(\rho^{\beta})_{t}\dot{u}_{x}^{2}|x|^{\frac{\alpha}{2}}dx-\frac{\alpha}{2}\int_{\R}\mu(\rho)\dot{u}_{x}|x|^{\frac{\alpha}{2}-2}x\dot{u}_{t}dx
-\int_{\R}\rho u\dot{u}_{x}|x|^{\frac{\alpha}{2}}\dot{u}_{t}dx\\
&-\int_{\R}((\rho^{\beta})_{t}u_{x}^{2}+2\mu(\rho)u_{x}u_{xt}+\beta(\rho^{\beta})_{t}u_{x}^{2}+2\beta\rho^{\beta}u_{x}u_{xt}-\gamma(\rho^{\gamma})_{t}u_{x}-\gamma\rho^{\gamma}u_{xt})\\
&\ \ \times (|x|^{\frac{\alpha}{2}}\dot{u}_{x}+\frac{\alpha}{2}|x|^{\frac{\alpha}{2}-2}x\dot{u})dx\\
\end{aligned}
\end{equation}
It follows from (\ref{c57}) that
\begin{equation}\label{c76}
\begin{aligned}
\frac{1}{2}\int_{\R}(\rho^{\beta})_{t}\dot{u}_{x}^{2}|x|^{\frac{\alpha}{2}}dx\leq
C\|(\rho^{\beta})_{t}\|_{L^{\infty}}\int_{\R}|x|^{\frac{\alpha}{2}}\dot{u}_{x}^{2}dx\leq
C\int_{\R}|x|^{\frac{\alpha}{2}}\dot{u}_{x}^{2}dx.
\end{aligned}
\end{equation}
Direct estimates lead to
\begin{equation}\label{c77}
\begin{aligned}
&-\frac{\alpha}{2}\int_{\R}\mu(\rho)\dot{u}_{x}|x|^{\frac{\alpha}{2}-2}x\dot{u}_{t}dx\\
&=-\frac{\alpha}{2}\frac{d}{dt}\int_{\R}\mu(\rho)\dot{u}_{x}|x|^{\frac{\alpha}{2}-2}x\dot{u}dx+\frac{\alpha}{2}\int_{\R}(\rho^{\beta})_{t}\dot{u}_{x}|x|^{\frac{\alpha}{2}-2}
x\dot{u}dx+\frac{\alpha}{2}\int_{\R}\mu(\rho)\dot{u}_{xt}|x|^{\frac{\alpha}{2}-2}x\dot{u}dx\\
&\leq
-\frac{\alpha}{2}\frac{d}{dt}\int_{\R}\mu(\rho)\dot{u}_{x}|x|^{\frac{\alpha}{2}-2}x\dot{u}dx+C\|(\rho^{\beta})_{t}\|_{L^{\infty}}\|\dot{u}_{x}\|_{L^{2}}\||x|^{\frac{\alpha}{2}-1}
\dot{u}\|_{L^{2}}+C\|\dot{u}_{xt}\|_{L^{2}}\||x|^{\frac{\alpha}{2}-1}\dot{u}\|_{L^{2}}\\
&\leq
-\frac{\alpha}{2}\frac{d}{dt}\int_{\R}\mu(\rho)\dot{u}_{x}|x|^{\frac{\alpha}{2}-2}x\dot{u}dx+C\||x|^{\frac{\alpha}{2}}\dot{u}_{x}\|_{L^{2}}
+C\||x|^{\frac{\alpha}{2}}\dot{u}_{x}\|_{L^{2}}\|\dot{u}_{xt}\|_{L^{2}}\\
&\leq
-\frac{\alpha}{2}\frac{d}{dt}\int_{\R}\mu(\rho)\dot{u}_{x}|x|^{\frac{\alpha}{2}-2}x\dot{u}dx+C(1+\|\dot{u}_{xt}\|_{L^{2}})\||x|^{\frac{\alpha}{2}}\dot{u}_{x}\|_{L^{2}},
\end{aligned}
\end{equation}
\begin{equation}\label{c78}
\begin{aligned}
-\int_{\R}\rho u\dot{u}_{x}|x|^{\frac{\alpha}{2}}\dot{u}_{t}dx&\leq C\int_{\R}(\sqrt{\rho}\dot{u}_{t}|x|^{\frac{\alpha}{4}})(\dot{u}_{x}|x|^{\frac{\alpha}{4}})\sqrt{\rho}udx\\
&\leq
C\|\sqrt{\rho}u\|_{L^{\infty}}\|\sqrt{\rho}\dot{u}_{t}|x|^{\frac{\alpha}{4}}\|_{L^{2}}\|\dot{u}_{x}|x|^{\frac{\alpha}{4}}\|_{L^{2}}\\
&\leq \varepsilon\int_{\R}\rho
\dot{u}_{t}^{2}|x|^{\frac{\alpha}{2}}dx+C\|\dot{u}_{x}|x|^{\frac{\alpha}{4}}\|_{L^{2}}^{2},
\end{aligned}
\end{equation}
and
\begin{equation}\label{c79}
\begin{aligned}
-&\int_{\R}\biggl((\rho^{\beta})_{t}u_{x}^{2}+2\mu(\rho)u_{x}u_{xt}+\beta(\rho^{\beta})_{t}u_{x}^{2}+2\beta\rho^{\beta}u_{x}u_{xt}-\gamma(\rho^{\gamma})_{t}u_{x}-\gamma\rho^{\gamma}u_{xt}\biggl)\\
&\ \ \ \ \biggl(|x|^{\frac{\alpha}{2}}\dot{u}_{x}+\frac{\alpha}{2}|x|^{\frac{\alpha}{2}-2}x\dot{u}\biggl)dx\\
&\leq C(\||x|^{\frac{\alpha}{2}}\dot{u}_{x}\|_{L^{2}}+\||x|^{\frac{\alpha}{2}-1}\dot{u}\|_{L^{2}})(\|(\rho^{\beta})_{t}\|_{L^{\infty}}\|u_{x}\|_{L^{4}}^{2}
+\|u_{x}\|_{L^{\infty}}\|u_{xt}\|_{L^{2}}\\
&\ \ \ +\|(\rho^{\beta})_{t}\|_{L^{\infty}}\|u_{x}\|_{L^{4}}^{2}+\|u_{x}\|_{L^{\infty}}\|u_{xt}\|_{L^{2}}+\|(\rho^{\gamma})_{t}\|_{L^{\infty}}\|u_{x}\|_{L^{2}}+\|\rho^{\gamma}\|_{L^{\infty}}\|u_{xt}\|_{L^{2}})\\
&\leq C\||x|^{\frac{\alpha}{2}}\dot{u}_{x}\|_{L^{2}}.
\end{aligned}
\end{equation}
Substituting $(\ref{c76})-(\ref{c79})$ into $(\ref{c75})$, we get
\begin{equation}\label{c80}
\begin{aligned}
\frac{d}{dt}&\int_{\R}\mu(\rho)\dot{u}_{x}^{2}|x|^{\frac{\alpha}{2}}dx+\int_{\R}\rho\dot{u}_{t}^{2}|x|^{\frac{\alpha}{2}}dx\\
&\leq \frac{d}{dt}\int_{\R}(\mu(\rho)u_{x}^{2}+\beta\rho^{\beta}u_{x}^{2}-\gamma\rho^{\gamma}u_{x})(|x|^{\frac{\alpha}{2}}\dot{u}_{x}
+\frac{\alpha}{2}|x|^{\frac{\alpha}{2}-2}x\dot{u})dx\\
&-\frac{\alpha}{2}\frac{d}{dt}\int_{\R}\mu(\rho)\dot{u}_{x}|x|^{\frac{\alpha}{2}-2}x\dot{u}dx+\int_{\R}|x|^{\frac{\alpha}{2}}\dot{u}_{x}^{2}dx
+C(1+\|\dot{u}_{xt}\|_{L^{2}})\||x|^{\frac{\alpha}{2}}\dot{u}_{x}\|_{L^{2}}\\
&\leq \frac{d}{dt}F(t)-\frac{d}{dt}H(t)
+C(1+\|\dot{u}_{xt}\|_{L^{2}}^{2}+\||x|^{\frac{\alpha}{2}}\dot{u}_{x}\|_{L^{2}}^{2}),
\end{aligned}
\end{equation}
where
$$F(t)=\int_{\R}(\mu(\rho)u_{x}^{2}+\beta\rho^{\beta}u_{x}^{2}-\gamma\rho^{\gamma}u_{x})(|x|^{\frac{\alpha}{2}}\dot{u}_{x}
+\frac{\alpha}{2}|x|^{\frac{\alpha}{2}-2}x\dot{u})dx,$$
$$H(t)=\frac{\alpha}{2}\int_{\R}\mu(\rho)\dot{u}_{x}|x|^{\frac{\alpha}{2}-2}x\dot{u}dx.$$
Multiplying the inequality $(\ref{c80})$ by $t$ and then integrating
the resulted inequality with respect to $t$ over $[\tau,t_{1}]$ with
both $\tau,t_{1}\in[0,T]$ give
\begin{equation}\label{c81}
\begin{aligned}
&t_{1}\|\sqrt{\mu(\rho)}|x|^{\frac{\alpha}{4}}\dot{u}_{x}\|_{L^{2}}^{2}(t_{1})+\int_{\tau}^{t_{1}}
t\|\sqrt{\rho}\dot{u}_{t}|x|^{\frac{\alpha}{4}}\|_{L^{2}}^{2}dt\\
&\leq
\tau\|\sqrt{\mu(\rho)}|x|^{\frac{\alpha}{4}}\dot{u}_{x}\|_{L^{2}}^{2}(\tau)+\int_{\tau}^{t_{1}}
\|\sqrt{\mu(\rho)}|x|^{\frac{\alpha}{4}}\dot{u}_{x}\|_{L^{2}}^{2}dt+t_{1}F(t_{1})-\tau
F(\tau)\\
&\ \ \ \ -\int_{\tau}^{t_{1}}F(t)dt-t_{1}H(t_{1})+\tau
H(\tau)+\int_{\tau}^{t_{1}}H(t)dt\\
&\ \ \ \
+C\int_{\tau}^{t_{1}}t(1+\|\dot{u}_{xt}\|_{L^{2}}^{2}+\||x|^{\frac{\alpha}{2}}\dot{u}_{x}\|_{L^{2}}^{2})dt.
\end{aligned}
\end{equation}
Note that
$$\int_{0}^{T}\|\sqrt{\mu(\rho)}\dot{u}_{x}|x|^{\frac{\alpha}{4}}\|_{L^{2}}^{2}dt\leq\int_{0}^{T}\|\sqrt{\mu(\rho)}|\dot{u}_{x}|(1+|x|^{\frac{\alpha}{2}})\|_{L^{2}}^{2}dt
\leq C(T),$$
\begin{equation}\nonumber
\begin{aligned}
\int_{0}^{T}F(t)dt&=\int_{0}^{T}\int_{\R}(\mu(\rho)u_{x}^{2}+\beta\rho^{\beta}u_{x}^{2}-\gamma\rho^{\gamma}u_{x})(|x|^{\frac{\alpha}{2}}\dot{u}_{x}
+\frac{\alpha}{2}|x|^{\frac{\alpha}{2}-2}x\dot{u})dxdt\\
&\leq
C\int_{0}^{T}\||x|^{\frac{\alpha}{2}}\dot{u}_{x}\|_{L^{2}}^{2}dt\leq
C(T),
\end{aligned}
\end{equation}
and
\begin{equation}\nonumber
\begin{aligned}
\int_{0}^{T}H(t)dt&=\int_{0}^{T}\int_{\R}\mu(\rho)\dot{u}_{x}|x|^{\frac{\alpha}{2}-2}x\dot{u}dxdt\leq C\int_{0}^{T}\int_{\R}|\dot{u}_{x}||x|^{\frac{\alpha}{2}-1}|\dot{u}|dxdt\\
&\leq
C\int_{0}^{T}\|\dot{u}_{x}\|_{L^{2}}\||x|^{\frac{\alpha}{2}-1}\dot{u}\|_{L^{2}}dt\\
&\leq
C\int_{0}^{T}(\|\dot{u}_{x}\|_{L^{2}}^{2}+\||x|^{\frac{\alpha}{2}}\dot{u}_{x}\|_{L^{2}}^{2})dt\leq
C(T).
\end{aligned}
\end{equation}
There exists a subsequence $\tau_{k}$ such that as $
k\rightarrow+\infty$,
$$\tau_{k}\rightarrow 0,\ \  \tau_{k}\biggl(\int_{\R}\mu(\rho)\dot{u}_{x}^{2}|x|^{\frac{\alpha}{2}}dx\biggl)(\tau_{k})\rightarrow0,\ \  \tau_{k}\biggl(\int_{\R}\mu(\rho)\dot{u}_{x}|x|^{\frac{\alpha}{2}-2}x\dot{u}dx\biggl)(\tau_{k})\rightarrow0,\ \
$$
and
$$\tau_{k}\biggl(\int_{\R}(\mu(\rho)u_{x}^{2}+\beta\rho^{\beta}u_{x}^{2}-\gamma\rho^{\gamma}u_{x})(|x|^{\frac{\alpha}{2}}\dot{u}_{x}
+\frac{\alpha}{2}|x|^{\frac{\alpha}{2}-2}x\dot{u})dx\biggl)(\tau_{k})\rightarrow0.$$
Taking $\tau=\tau_{k}$ in $(\ref{c81})$, then letting
$k\rightarrow+\infty$ and using the Cauchy inequality and Gronwall
inequality, one can obtain
$$t\|\sqrt{\mu(\rho)}|x|^{\frac{\alpha}{4}}\dot{u}_{x}\|_{L^{2}}^{2}+\int_{0}^{T}t\|\sqrt{\rho}\dot{u}_{t}|x|^{\frac{\alpha}{4}}\|_{L^{2}}^{2}dt\leq C(T).
$$
The proof of the lemma  is completed.
\end{proof}
\end{Lemma}

\begin{Lemma}\label{lemmac15}
Let $(\rho,u)$  be the smooth solution to $(\ref{a1})-(\ref{a4})$.
Then for any $0\leq t\leq T$, it holds
$$t\int_{\R}u_{xxx}^{2}dx\leq C(T).$$
\begin{proof}
Rewriting the equation $(\ref{a1})_{2}$ as
$$\rho\dot{u}+(\rho^{\gamma})_{x}=\mu(\rho)u_{xx}+(\rho^{\beta})_{x}u_{x}.$$
Differentiating the above equation with respect to $x$, we have
$$\mu(\rho)u_{xxx}=\rho_{x}\dot{u}+\rho\dot{u}_{x}+(\rho^{\gamma})_{xx}-
(\rho^{\beta})_{xx}u_{x}-2(\rho^{\beta})_{x}u_{xx}.$$
This implies
that
\begin{equation}\nonumber
\begin{aligned}
\|u_{xxx}\|_{L^{2}}&\leq C\|\rho_{x}\dot{u}+\rho\dot{u}_{x}+(\rho^{\gamma})_{xx}-(\rho^{\beta})_{xx}u_{x}-2(\rho^{\beta})_{x}u_{xx}\|_{L^{2}}\\
&\leq C(\|\dot{u}\|_{L^{\infty}}\|\rho_{x}\|_{L^{2}}+1)\leq C(\|\dot{u}\|_{L^{\infty}}+1)\\
&\leq C(\|\dot{u}\|_{L^{\frac{4}{\alpha-2}}}+\|\dot{u}_{x}\|_{L^{2}}+1)\\
&\leq C(\||x|^{\frac{\alpha}{4}}\dot{u}_{x}\|_{L^{2}}+1).
\end{aligned}
\end{equation}
Thanks to Lemma $\ref{lemmac14}$, it deduces
$$t\int_{\R}u_{xxx}^{2}dx\leq C(T).$$
The proof of the lemma  is completed.
\end{proof}
\end{Lemma}

\section{Proof of Theorem $\ref{theorem}$.}

In this section, we give the proof of our main result.

We first state the local existence and uniqueness of classical
solution when the initial  data may contain vacuum, the proof is
referred to \cite{LZ2012}, \cite{CK2006}.

\begin{Lemma}\label{lemmab1}
Under assumptions of Theorem $(\ref{theorem})$, there exists a $T_{*}>0$ and an unique classical solution $(\rho,u)$ to the Cauchy problem $(\ref{a1})-(\ref{a4})$ satisfying $(\ref{a7})$ with $T$ replaced by $T_{*}$.
\end{Lemma}

With all the a priori estimates in Section 2 at hand, we are ready to prove the main results of this paper in this section.

\begin{proof}[Proof of Theorem $\ref{theorem}$]
We first show that $(\rho,u)$ is a classical solution to
$(\ref{a1})-(\ref{a4})$ if $(\rho,u)$ satisfies $(\ref{a7})$. Since
$u\in L^{2}(0,T;L^{\frac{2}{\alpha-1}}\cap D^{3}(\R))\ and\ u_{t}\in
L^{2}(0,T;L^{\frac{2}{\alpha-1}}\cap D^{1}(\R))$, the Sobolev's
embedding theorem implies that
$$u\in C([0,T];L^{\frac{2}{\alpha-1}}\cap D^{2}(\R))\hookrightarrow C([0,T]\times \R)$$
It follows from $(\rho,p(\rho))\in L^{\infty}(0,T;H^{2}(\R))$, and
$(\rho,P(\rho))_{t}\in L^{\infty}(0,T;H^{1}(\R))$ that
$(\rho,P(\rho))\in C([0,T];W^{1,q}(\R))\cap C([0,T];H^{2}(\R)-weak)$
with $2< q< +\infty$. This, together with $(\ref{a1})_{1}$ and
\cite{CK2006}, implies that
$$(\rho,P(\rho))\in C([0,T];H^{2}(\R))$$
Moreover, since for any $\tau\in (0,T)$
$$u\in L^{\infty}(\tau,T;L^{\frac{2}{\alpha-1}}\cap D^{3}(\R))\ ,\ u_{t}\in L^{\infty}(\tau,T;L^{\frac{2}{\alpha-1}}(\R)),$$
we have
$$u\in C([\tau,T];W^{3,q}(\R))\hookrightarrow C([\tau,T];C^{2,\zeta}(\R)),\ 0<\zeta<\frac{1}{2}\ for\ 1<q<2.$$
Note that
$$(\rho_x,(P(\rho))_x)\in C([0,T];H^{1}(\R))\hookrightarrow C([0,T]\times \R).$$
Using the continuity equation $(\ref{a1})_{1}$, one has
$$\rho_{t}=-(\rho u_{x}+u\rho_{x})\in C([\tau,T]\times \R).$$
Using the momentum equation $(\ref{a1})_{2}$, one has
\begin{equation*}
\begin{aligned}
(\rho u)_{t}&=(\mu(\rho)u_{x})_{x}-(\rho u^{2})_{x}-p_{x}\\
&=\mu(\rho)u_{xx}+(\rho^{\beta})_{x}u_{x}-\rho_{x}u^{2}-2\rho uu_{x}-(\rho^{\gamma})_{x}\\
&\in C([\tau,T]\times \R)
\end{aligned}
\end{equation*}

 Theorem $\ref{theorem}$ follows from Lemma $\ref{lemmab1}$ which is about
the local well-posedness of the classical solution and global (in
time) a priori estimates in Section 2. In fact, by Lemma
$\ref{lemmab1}$, there exists a local classical solution $(\rho,u)$
on the time interval $(0,T_{*}]$ with $T_{*}>0$. Now let $T^{*}$ be
the maximal existing time of the classical solution $(\rho,u)$ in
Lemma $\ref{lemmab1}$. Then obviously one has $T^{*}\geq T_{*}$. Now
we claim that $T^{*}\geq T$\ for any $T>0$ being any fixed positive
constant given in Theorem $\ref{theorem}$. Otherwise, if $T^{*}< T$,
then all the a priori estimates in Section 2 hold with $T$ being
replaced by $T^{*}$. In particular, from the Lemma $\ref{lemmac7},\
\ref{lemmac9}$, it holds that
$$(1+|x|^{\frac{\alpha}{2}})\sqrt{\rho}\dot{u}\in C([0,T^{*}];L^{2}(\R)).$$
Therefore, it follows from a  priori estimates in Section 2  that
$(\rho ,u)(x,T^{*})$ satisfy $(\ref{a7})$ and the compatibility
condition $(\ref{a6})$ at time $t=T^{*}$ with
$g(x)=\sqrt{\rho}\dot{u}(x,T^{*})$. By using lemma $\ref{lemmab1}$
again, there exists a $T^{*}_{1}>0$ such that the classical solution
$(\rho,u)$ in Lemma $\ref{lemmab1}$ exists on $(0,T^{*}+T^{*}_{1}]$,
which contradicts with $T^{*}$ being the maximal existing time of
the classical solution $(\rho,u)$. Thus it holds that $T^{*}\geq T$.
The proof of Theorem $\ref{theorem}$ is completed.
\end{proof}




\vspace{2mm}


\end{document}